\documentclass[a4paper, 11pt]{amsart}

\usepackage{amsmath}
\usepackage{amsfonts}
\usepackage{amssymb}
\usepackage{mathrsfs}                  
\usepackage{amsthm}
\usepackage{enumitem}
\usepackage{amscd}
\usepackage{xcolor}
\usepackage[hyperfootnotes=false]{hyperref}
\usepackage{mathtools}
\usepackage{geometry}
\usepackage{bbm}

\newcommand\blfootnote[1]{%
  \begingroup
  \renewcommand\thefootnote{}\footnote{#1}%
  \addtocounter{footnote}{-1}%
  \endgroup
}

\numberwithin{equation}{section}
\theoremstyle{plain}
\newtheorem{theorem}{Theorem}[section]
\newtheorem{proposition}[theorem]{Proposition}
\newtheorem*{SEP}{(SEP)}
\newtheorem*{Classical SEP}{(Classical SEP)}
\newtheorem{lemma}[theorem]{Lemma}

\newtheorem{remark}[theorem]{Remark}
\newtheorem{definition}[theorem]{Definition}
\newtheorem{example}[theorem]{Example}
\newtheorem{assumption}[theorem]{Assumption}

\newcommand{\dd}{\,\mathrm{d}}

\newcommand{\E}{\mathbb{E}}

\newcommand{\R}{\mathbb{R}}

\newcommand{\thor}{1}

\renewcommand{\d}{\mathrm{d}}
\renewcommand{\P}{\mathbb{P}}

\title{On Skorokhod Embeddings and Poisson Equations}

\author[D\"oring]{Leif D\"oring}
\address{Leif D\"oring, University of Mannheim, Germany}
\email{doering@uni-mannheim.de}

\author[Gonon]{Lukas Gonon}
\address{Lukas Gonon, Universit\"at Sankt Gallen, Switzerland}
\email{lukas.gonon@unisg.ch}

\author[Pr\"omel]{David J. Pr\"omel}
\address{David J. Pr\"omel, University of Oxford, United Kingdom}
\email{proemel@maths.ox.ac.uk}

\author[Reichmann]{Oleg Reichmann}
\address{Oleg Reichmann, European Investment Bank, Luxembourg}
\email{o.reichmann@eib.org}

\date{\today}

\begin{document}

\begin{abstract} 
  The classical Skorokhod embedding problem for a Brownian motion~$W$ asks to find a stopping time~$\tau$ so that $W_\tau$ is distributed according to a prescribed probability distribution~$\mu$. Many solutions have been proposed during the past 50 years and applications in different fields emerged. This article deals with a generalized Skorokhod embedding problem (SEP): Let $X$ be a Markov process with initial marginal distribution $\mu_0$ and let $\mu_1$ be a probability measure. The task is to find a stopping time~$\tau$ such that $X_\tau$ is distributed according to $\mu_1$. More precisely, we study the question of deciding if a finite mean solution to the SEP can exist for given $\mu_0, \mu_1$ and the task of giving a solution which is as explicit as possible.\\
  If $\mu_0$ and $\mu_1$ have positive densities $h_0$ and $h_1$ and the generator $\mathcal A$ of $X$ has a formal adjoint operator $\mathcal A^*$, then we propose necessary and sufficient conditions for the existence of an embedding in terms of the Poisson equation $\mathcal A^* H=h_1-h_0$  and give a fairly explicit construction of the stopping time using the solution of the Poisson equation. For the class of L\'evy processes we carry out the procedure and extend a result of Bertoin and Le Jan to L\'evy processes without local times.
\end{abstract}

\maketitle
\frenchspacing\blfootnote{The views expressed in this article are those of the authors and not necessarily of the European Investment Bank.}

\noindent\textbf{Key words and phrases:} Fokker-Planck equation, L\'evy process, Markov process, Skorokhod embedding problem, random time-change.\\
\textbf{MSC 2010 Classification:} 60G40, 60J75.



\section{Introduction and Main Results}

The Skorokhod embedding problem was originally formulated and solved by Skorokhod~\cite{Skorokhod1961,Skorokhod1965} for a one-dimensional Brownian motion~$W$ started from~$0$ and a given probability measure~$\mu$:
\begin{Classical SEP}
  $\text{Find a stopping time } \tau\text{ such that } W_\tau \sim \mu \text{ and } \mathbb{E}[\tau] < \infty. $
\end{Classical SEP}
The additional requirement on $\tau$ to satisfy $\E[\tau]<\infty$ is commonly posed to exclude non-meaningful solutions. As observed by Doob (see \cite[Remark~51.7]{Rogers2000v2}) without this condition a trivial solution would be $\tau=\inf\{ t \geq 2: B_t=F_\mu^{-1}(\Phi(B_1))\}$, where $\Phi$ is the distribution function of a standard normal variable and $F^{-1}_\mu$ is the right-inverse of the distribution function $F_\mu$ of $\mu$. There is a great ongoing effort to obtain solutions with different properties to the Skorokhod embedding problem in different generalizations. For a survey paper on classical results we refer to~\cite{Obloj2004} and references therein.\smallskip

Recent motivation to deal with various versions of the classical Skorokhod embedding problem stems from its applications in mathematical finance starting with the seminal work of Hobson~\cite{Hobson1998}, where model-independent pricing bounds and hedging techniques for lookback options were studied by means of Skorokhod embedding. The link between robust financial mathematics and the classical SEP was utilized by many authors to determine robust price bounds for exotic options, see \cite{Hobson2011} for a more detailed introduction to this area. More recently, additional interest in the Skorokhod embedding problem was also caused by new applications in game theory (e.g. \cite{Seel2013,Feng2016}) and in numerical analysis (e.g. \cite{Gassiat2015,Ankirchner2016}).\smallskip

There are two direct extensions of the Skorokhod embedding problem: generalizing the process and generalizing the deterministic initial condition $\delta_0$ to an arbitrary distribution $\mu_0$. A natural motivation for the latter is the interest of constructing multi-marginal Skorokhod embeddings.\smallskip

The version of the Skorokhod embedding problem we deal with allows for a general process and a general initial distribution. Let $\mu_0$ and $\mu_1$ be two given probability distributions. On a complete probability space $(\Omega,\mathcal{F},\P^{\mu_0})$ we consider a stochastic process $L$ with $L_0 \sim \mu_0$ under $\P^{\mu_0}$ and denote by $(\mathcal F_t)_{t\geq 0}$ the $\P^{\mu_0}$-augmented natural filtration of $L$. This setting leads to the following formulation of the Skorokhod embedding problem:
\begin{SEP}
  Find an $(\mathcal{F}_t)_{t \geq 0}$-stopping time $\tau$ such that $ L_\tau \sim \mu_1$ and $\mathbb{E}^{\mu_0}[\tau] < \infty.$ 
\end{SEP}
The first natural question is under which conditions an embedding $\tau$ exists. For a Brownian motion starting from an initial law~$\mu_0$ with finite second moment this is a classical result: there is a finite mean embedding for $\mu_1$ if and only if $\mu_0$ and $\mu_1$ have the same first moment, finite second moments and $\mu_0$ is smaller than $\mu_1$ in convex order, i.e.
\begin{align*}
  \int_\mathbb{R} \varphi(x)\, \mu_0(\d x) \leq \int_\mathbb{R} \varphi(x)\, \mu_1(\d x) \quad\text{ for all } \varphi \text{ convex. }  
\end{align*}
Sufficiency follows e.g. by \cite{baxter1974}, necessity by the optional sampling theorem and Jensen's inequality.
To give the right generalization of this property for more general Markov processes is the main purpose of this article. Using general Markov process theory we find an abstract formulation in terms of Poisson equations which becomes explicit for L\'evy processes but we believe to hold much more generally.\smallskip

Recall that a continuous-time process $(L_t)_{t\geq0}$ with values in $\mathbb{R}$ is called L\'evy process if it has almost surely RCLL sample paths, is almost surely issued from $0$, is stochastically continuous and has stationary and independent increments. Due to the L\'evy-Khintchine representation, there exist $\alpha \geq 0$, $\gamma \in \mathbb{R}$ and a measure $\nu$ on $\mathbb{R}$ with $\nu(\{0\}) = 0$ and $\int_\mathbb{R} (x^2 \wedge 1)\, \nu (\d x) < \infty$ such that
\begin{equation}\label{eq:charFctLevy}
  \mathbb{E}[e^{i u L_t}] = e^{t \eta(u)}, \quad u \in \mathbb{R},\, t \geq 0,
\end{equation}
with the characteristic exponent
\begin{equation}\label{eq:symbol} 
  \eta(u) = - \frac{1}{2}\alpha^2 u^2 + i u \gamma + \int_{\mathbb{R}\setminus \{0\}} (e^{i u y} - 1 - i u y \mathbbm{1}_{\{|y|\leq 1\}} )\,\nu(\d y),  \quad u \in \mathbb{R}.
\end{equation}
The triplet $(\alpha^2,\gamma,\nu)$ is called L\'evy triplet and fully characterizes $L$. We exclude the trivial case of a constant L\'evy process, i.e. $\alpha=\gamma=\nu=0$. For more background information we refer for instance to the introductory texts of Bertoin~\cite{Bertoin1996} and Kyprianou~\cite{Kyprianou2014}. 
A L\'evy process with initial distribution $\mu_0$ is defined as $ L=\tilde L+X_0$, where $X_0\sim \mu_0$ is independent from the L\'evy process $\tilde L$. Throughout the article, $L$ will be a L\'evy process with initial distribution $\mu_0$ under $\P^{\mu_0}$. For the special case $\mu_0=\delta_0$ we always abbreviate $\P=\P^{\delta_0}$.\smallskip


To the best of our knowledge there is only one article that deals with sufficient \textit{and} necessary conditions for the existence of finite mean Skorokhod embeddings for L\'evy processes; for the particular case $\mu_0=\delta_0$. For symmetric and recurrent L\'evy processes that possess jointly continuous local times (e.g. $\alpha$-stable processes with $\alpha\in (1,2]$), Bertoin and Le Jan~\cite{Bertoin1992} give the following necessary and sufficient condition for the Skorokhod embedding problem: If $\widehat{\mu_1}$ denotes the Fourier transform of the measure $\mu_1$, then the necessary and sufficient condition for the existence of a finite mean Skorokhod embedding is
\begin{align}\label{BJ}
  \frac{\widehat{\mu_1} - 1}{\eta} \in L^1(\R), \quad H\geq 0 \quad \text{and}\quad H\in L^1(\R),
\end{align}
where
\begin{align}\label{bb}
  H(x):= \frac{1}{2 \pi} \int_\mathbb{R} \frac{\widehat{\mu_1}(\xi)-1}{\eta(\xi)}e^{-i x \xi} \dd \xi,\quad x \in \mathbb{R}.
\end{align}

We should mention that the results of~\cite{Bertoin1992} hold not only for L\'evy processes but also for Hunt processes with local times. Since their proofs were based on excursion theory, we had to develop a completely different approach to deal with processes that do not have local times (e.g. the Cauchy process).\smallskip
  
The main result of the present article shows that the obvious generalization of~\eqref{BJ} and~\eqref{bb} to non-trivial $\mu_0$ (i.e. replacing $1$ by $\widehat \mu_0$) is the necessary and sufficient condition also for a wide class of measures and L\'evy processes without local times. Allowing the L\'evy process to be more general forces us on the other hand to assume a priori regularity on $\mu_0, \mu_1$. We will always assume that $\mu_0, \mu_1$ have positive densities with respect to the Lebesgue measure. Additional smoothness will be imposed (e.g. $h_0, h_1\in C_0(\R)$ for the Brownian motion). Assumptions on the densities are different for different L\'evy processes; we state the precise assumptions in Section~\ref{sec:smoothness} below.
 
\begin{theorem}\label{thm:main}
  Suppose $L$ is a L\'evy process with initial distribution $\mu_0$ and characteristic exponent $\eta$. Suppose $\mu_0, \mu_1$ have strictly positive densities $h_0, h_1$ which are ``sufficiently smooth'' (specified below in Assumption~\ref{ass:regularity}). 
  \begin{itemize}
    \item[(i)] The necessary and sufficient condition for the existence of a finite mean Skorokhod embedding is 
      \begin{align}\label{BJ2}
        \frac{\widehat{\mu_1} - \widehat{\mu_0}}{\eta} \in L^1(\R), \quad H\geq 0 \quad \text{and}\quad H\in L^1(\R),
      \end{align}
      where
      \begin{align}\label{b}
        H(x):= \frac{1}{2 \pi} \int_\mathbb{R} \frac{\widehat{\mu_1}(\xi)-\widehat{\mu_0}(\xi)}{\eta(\xi)}e^{-i x \xi} \dd \xi,\quad x \in \mathbb{R}.
      \end{align}
    \item[(ii)] If \eqref{BJ2} is satisfied, then an explicit solution under $\P^{\mu_0}$ is given as follows:  
    \begin{equation*}
      \tau := \inf\bigg\lbrace t \in [0,\rho) \, :\, \int_0^t e^{-G(r)} \frac{h_1(L_r)}{H(L_r)} \dd r  \geq 1 \bigg\rbrace \wedge \rho,
    \end{equation*}
    where, for $t\geq 0$,
    \begin{equation*}
      \rho :=  \inf\{ t \in [0,\infty) \, : \, H(L_t) = 0 \}\quad\text{and}\quad \text  G(t)  := \int_0^t \frac{h_1(L_r)-h_0(L_r)}{H(L_r)} \dd r
    \end{equation*}
    with the usual convention $\inf \emptyset := \infty$. 
    \item[(iii)] With $\tau$ from (ii) it holds that $\E^{\mu_0}[\tau]=\int_\mathbb{R} H(x) \dd x$. 
  \end{itemize}
\end{theorem}

The conditions might look complicated at first sight but they are explicit since they only involve the given densities and the given characteristic exponent of the L\'evy process. Also the stopping time is fairly explicit: it only involves the process and explicit functions but no further stochastic quantities (e.g. local times).\smallskip

Even though the three conditions in~\eqref{BJ2} cannot be considered separately from each other, each of them has an interpretation in analogy to the Brownian case mentioned above: Since $\eta(0)=0$, the integrability at zero of $(\widehat{\mu_1} - \widehat{\mu_0})/\eta$ forces a decay of $\widehat{\mu_1} - \widehat{\mu_0}$ in relation to the behavior of $\eta$ at zero. Since the behavior at zero of a characteristic function relates to the moments, the integrability of $(\widehat{\mu_1} - \widehat{\mu_0})/\eta$ is an abstract condition for equal first moments of $\mu_0, \mu_1$. Non-negativity of $H$ is a generalization of the convex order condition for Brownian motion and integrability of $H$ corresponds to finite second moments. \smallskip

\begin{remark}
  Note that for lattice type L\'evy processes there exist $u_0 \neq 0$ with $\eta(u_0)=0$. For such $u_0$ the condition $(\widehat{\mu_1} - \widehat{\mu_0})/\eta \in L^1(\R)$ in \eqref{BJ2} thus requires a decay of $\widehat{\mu_1}(u) - \widehat{\mu_0}(u)$ as $u \to u_0$.
\end{remark}

While Bertoin and Le Jan~\cite{Bertoin1992} deal with necessary and sufficient conditions for the solvability of the Skorokhod embedding problem for, in particular, certain L\'evy processes and $\mu_0=\delta_0$ (see above), sufficient conditions for different types of L\'evy processes and $\mu_0=\delta_0$ were also obtained in \cite{Monroe1972} and \cite{Obloj2009}. Namely, Monroe~\cite{Monroe1972} adresses symmetric $\alpha$-stable L\'evy processes with $\alpha \in (1,2]$ and Ob{\l}{\'o}j and Pistorius~\cite{Obloj2009} the case of spectrally negative L\'evy processes. In a more abstract setting Falkner and Fitzsimmons~\cite{Falkner1991} provide even necessary and sufficient conditions for general but transient Markov processes, which cover only partially the class of L\'evy processes. For a relaxed version of the SEP (allowing for randomized stopping times, i.e. allowing for stopping times which are measurable w.r.t. a larger filtration than the natural filtration generated by the underlying Markov process) Rost~\cite{Rost1971} shows necessary and sufficient conditions for general Markov processes. A discussion about differences between randomized and non-randomized solutions to the SEP can be found for instance in~\cite{Falkner1991}.

\begin{remark}\label{remark}
  All results about the Skorokhod embedding problem are presented for L\'evy processes for the sake of clarity. However, we believe that most arguments can be extended to more general Markov processes under rather unhandy conditions. 
  The main finding of this article reveals a direct connection between the Skorokhod embedding problem and the existence and positivity/integrability of solutions to the Poisson equation 
  \begin{align}\label{eq:poisson eq}
    \mathcal A^* H = h_1-h_0,
  \end{align}
  where $\mathcal A^*$ denotes (if it exists) the formal adjoint operator of the generator $\mathcal A$ of the given Markov process $L$. A sketch is given in Section~\ref{sec:sketch} to explain why the existence of a positive and integrable solution to the Poisson equation contains exactly the information needed for the finite mean Skorokhod embedding problem with densities $h_0$ and $h_1$.
\end{remark}

In contrast to Remark~\ref{remark} the statement of Theorem~\ref{thm:main} involves explicit quantities instead of solutions to Poisson equations. This, indeed, is a speciality for L\'evy processes for which the Poisson equation can be solved in Fourier space: To see this we recall the Fourier representation $\widehat{\mathcal A^* H}(u)=\eta(u) \widehat{H}(u)$ of $\mathcal A^*$, where $\mathcal A^*$ is the generator of the dual L\'evy process $-L$. To solve \eqref{eq:poisson eq} one takes Fourier transforms of both sides, solves as $\widehat{H}=(\widehat{h_1}-\widehat{h_0})/\eta$ and takes the inverse Fourier transform. This gives exactly the form of $H$ given in Theorem~\ref{thm:main}. Taking the inverse Fourier transform is valid thanks to the property $\widehat{H}=(\widehat{h_1}-\widehat{h_0})/\eta\in L^1$. This analysis in the context of L\'evy processes is carried out in Subsection~\ref{poisson}, see in particular Proposition \ref{p1}.

\begin{remark}
Let us compare our results to those of \cite{Rost1970} in more detail. Define the measures $\mu_i U$ by 
\[ \mu_i U(A) := \E^{\mu_i} \left[ \int_0^\infty 1_{A}(L_t) \d t \right], \quad A \in \mathcal{B}(\R). \]
Under the assumption that $\mu_0 U$ is a $\sigma$-finite measure (which is a transience hypothesis, see \cite{Falkner1991}), 
\cite{Rost1970} proves that a (randomized) stopping time $\bar{\tau}$ such that $L_{\bar{\tau}} \sim \mu_1$ under $\P^{\mu_0}$ exists if and only if $\mu_0 U \geq \mu_1 U$. 
Since a solution to (SEP) is in particular also a randomized stopping time, the necessary and sufficient condition \eqref{BJ2} in Theorem~\ref{thm:main} implies $\mu_0 U \geq \mu_1 U$. 
This property appears to correspond to non-negativity of $H$, whereas the other two conditions in \eqref{BJ2} seem to correspond to the additional restrictions posed on the stopping time in (SEP), 
namely that it should be non-randomized (which is implicit in our formulation) and have finite expected value.
\end{remark}

\subsection{Regularity Assumptions}\label{sec:smoothness}

For different L\'evy processes the necessary and sufficient conditions for the solvability of the Skorokhod embedding problem (SEP) provided in Theorem~\ref{thm:main} require different regularity assumptions on the initial density $h_0$ and the target density $h_1$. In order to state these assumptions, we distinguish between the following types of L\'evy processes.

\begin{definition}
  We say a L\'evy process with characteristic exponent $\eta$ is of type
  \begin{itemize}
    \item[S] if it is symmetric and $\int_1^\infty \frac{1}{|\eta(u)|}\dd u < \infty$,
    \item[0] if $\liminf_{u\to \infty} |\eta(u)|\in (0,\infty]$,
    \item[D] if $\liminf_{u\to \infty} |\eta(u)|=0$.
  \end{itemize}
\end{definition}

Notice that these three types cover all L\'evy processes as in particular any L\'evy process is either of type 0 or of type D. Based on this classification, we make the following regularity assumptions on the densities $h_0, h_1$.

\begin{assumption}[Regularity Assumptions]\label{ass:regularity}~
  \begin{itemize}
    \item If $L$ is of type~S, then $h_0, h_1 \in C_0(\R)$.
    \item If $L$ is of type~0, then $h_i \in C_0^2(\mathbb{R})$ with $h_i^{(k)} \in L^1(\mathbb{R})$ for $k=1,2$, $i=0,1$, where $h_i^{(k)}$ is the $k$-th derivative of $h_i$.
    \item If $L$ is of type~D, then $\widehat{h_1}-\widehat{h_0} \in C_c(\R)$. 
    \end{itemize}
\end{assumption}

The L\'evy processes considered by Bertoin and Le Jan are of type 0 as we will see in the next remark. The subclass of processes considered in the Appendix of \cite{Bertoin1992} (for which conditions \eqref{BJ} were proved) are even of type S.

\begin{remark}
  In \cite{Bertoin1992}, the L\'evy processes are assumed to be recurrent and satisfy that $0$ is regular for $0$. Excluding the compound Poisson case, the last condition is equivalent to condition~(i) of Lemma~\ref{lem:levyass} below and
  \begin{equation*}
    \sigma^2 > 0 \quad \text{ or } \quad \int_\R (|x| \wedge 1)\,\nu(\d x) = \infty,
  \end{equation*}
  see~\cite[Section I.30]{Rogers2000}. Hence, these processes are of type~0 by Lemma~\ref{lem:levyass}.
\end{remark}

Let us give further examples:

\begin{example}
  \textup{Type~S}: Symmetric $\alpha$-stable L\'evy processes with index $\alpha \in (1,2]$ are of type~S. Indeed, for such processes one has $\eta(u)=-c|u|^{\alpha}$ for some $c>0$ and so $\int_1^\infty |\eta(u)|^{-1} \d u = (\alpha -1)/c  < \infty$ is satisfied. In particular, a Brownian motion is of type~S and so for a Brownian motion Theorem~\ref{thm:main} provides a solution to the SEP for any positive, continuous densities $h_0, h_1\in C_0(\R)$ which have the same first moment, finite second moments and which are in convex order.  
  
  \textup{Type~0}: Symmetric $\alpha$-stable L\'evy processes with index $\alpha \in (0,1]$ are of type~0, but not of type~S.
  
  \textup{Type~D}: Lattice-type compound Poisson processes are of type~D. Other examples of processes of type~D can be found in \cite[Exercise~I.7]{Bertoin1996} and \cite[Example~41.23]{Sato1999}.
\end{example}

In fact, L\'evy processes of type 0 form a large class as demonstrated by the sufficient conditions presented in the next lemma. 

\begin{lemma}\label{lem:levyass} 
  Suppose that either
  \begin{itemize}
    \item[(i)] $\int_\R \mathrm{Re}\left(\frac{1}{1-\eta(\xi)}\right) \dd \xi  < \infty$ or 
    \item[(ii)] for some $t > 0$, the distribution of $L_t-L_0$ has a non-trivial absolutely continuous part.
  \end{itemize}
  Then $L$ is of type 0.
\end{lemma}

\begin{proof}
  We argue by contraposition. Suppose there exists $\{u_k\}_{k \in \mathbb{N}} \subset \mathbb{R}$ such that $\lim_{k \to \infty} |u_k| = \infty$ and  $\lim_{k \to \infty}\eta(u_k) =  0$. Then condition (ACP) in \cite{Sato1999} cannot be satisfied by the argument used in \cite[Example~41.23]{Sato1999}, which we reproduce here for convenience: denoting for $q>0$ by $V^q$ the $q$-potential measure 
  \[ V^q(B)= \E \left[ \int_0^\infty e^{-q t} 1_{B}(L_t) \d t \right], \quad B \in \mathcal{B}(\R),  \]
  one may use \cite[Prop.~37.4]{Sato1999} to obtain
  \[ \lim_{k \to \infty} \widehat{q V^q}(u_k) = \lim_{k \to \infty} \frac{q}{q-\eta(u_k)} = 1 = q V^q(\R). \]
  Since $\lim_{k \to \infty} |u_k| = \infty$, this shows that  $\widehat{q V^q}$ does not vanish at infinity and so by the Riemann-Lebesgue Theorem $q V^q$ does not have an absolutely continuous part. Thus, condition (ACP) in \cite{Sato1999} is indeed not satisfied. Combining this with \cite[Thm.~43.3]{Sato1999} and \cite[Remark~43.6]{Sato1999}, condition~(i) does not hold. Similarly, for any $t > 0$, $\lim_{k \to \infty} \E^{\mu_0}[e^{i u_k (L_t-L_0)}] = 1$ and by the Riemann-Lebesgue Theorem it follows that the law of $L_t-L_0$ does not have an absolutely continuous part. Hence, (ii) does not hold either.
\end{proof}

\subsection{Sketch of the Proof}\label{sec:sketch}

For the convenience of the reader we give a brief sketch of the arguments to explain why the existence of a finite mean Skorokhod embedding for a Markov process is related to the existence of a non-negative and integrable solution to the Poisson equation $\mathcal A^* H=h_1-h_0$. The link to the theorem then comes from explicit solvability of the Poisson equation in the case of L\'evy processes as explained below Remark \ref{remark}.

\subsubsection{Necessity of the Conditions} 

Suppose there is a finite mean Skorokhod embedding $\tau$ for a Markov process with generator $\mathcal A$ and adjoint $\mathcal A^*$ for the measures $\mu_i(\d x)=h_i(x)\dd x$. Since $\tau$ has finite mean, Dynkin's formula yields\footnote{See Section~\ref{subsec:levyprelim} for the definition of $D(\mathcal A) \subset C_0(\R)$ and note e.g.\ $C_c^2(\R) \subset D(\mathcal{A})$.}
\begin{align*}
  \E^{\mu_0} \left[ \int_0^{\tau} \mathcal{A} f(L_s) \dd s \right] = \E^{\mu_0}[f(L_{\tau})] - \E^{\mu_0}[f(L_0)],\quad f\in  D(\mathcal A).
\end{align*}
Let us assume for the moment there is a solution $H$ to the Poisson equation $\mathcal A^*H=h_1-h_0$. Integrating out the assumed distributions $L_\tau\sim \mu_1$ and $L_0\sim \mu_0$ and then switching to the adjoint operator gives
\begin{align*}
  \E^{\mu_0} \left[ \int_0^{\tau} \mathcal{A} f(L_s) \dd s \right] = \int_{\R} f(x)(h_1(x)-h_0(x))\dd x=\int_{\R} f(x) \mathcal A^* H(x)\dd x  =\int_{\R} \mathcal A f(x) H(x)\dd x.
\end{align*}
Now suppose the range of $\mathcal A$ is rich enough to approximate any positive test function with compact support and constant functions, then one obtains integrability and non-negativity for $H$: the first as $\mathcal A f\equiv 1$ gives $\int_\R  H(x)\dd x=\E^{\mu_0} \left[\tau \right] <\infty$ and the second since the left-hand side is non-negative whenever $\mathcal A f\ge 0$. Hence, if the existence of a finite mean Skorokhod embedding also implies solvability of the Poisson equation, then the solution $H$ is necessarily positive and integrable.\smallskip

For the existence of $H$ properties of $\mathcal A^*$ need to be studied in detail. In the case of a L\'evy process we prove (under regularity assumptions on $h_0, h_1$) that the existence of a finite mean Skorokhod embedding implies $(\widehat{\mu_1} -\widehat{\mu_0} )/\eta \in L^1(\R)$, from which it then follows (see the discussion below Remark~\ref{remark}) that a solution $H$ to the Poisson equation exists and is given by the inverse Fourier transform of $(\widehat{\mu_1} -\widehat{\mu_0} )/\eta$.
 
\begin{remark}
  The heart of the argument is the use of Dynkin's formula which does not assume the underlying process to be L\'evy. For the existence of the solution $H$ to the Poisson equation our argument does not extend to more general Markov processes. Nonetheless, we do believe the Poisson equation for more general processes is solvable as soon as there is a finite mean solution to the Skorokhod embedding problem.
\end{remark}

\subsubsection{Sufficiency of the Conditions}\label{sec:summary}

Assume there is a solution $H$ to the Poisson equation~\eqref{eq:poisson eq} which is non-negative and integrable. The approach presented in this article is inspired by Bass' solution \cite{Bass1983} to the classical Skorokhod embedding problem for Brownian motion. While the construction of Bass is short and elegant due to particular properties of Brownian motion, our variant requires more machinery. 
\smallskip

To start with let us first recall Bass' strategy. Let $W$ be a Brownian motion and $\mu$ be a centered probability measure on $\R$ with finite second moment. Bass' construction, slightly reinterpreted, can be split into two steps:\smallskip

\textbf{Step~1:} The first observation is that there is a function $g\colon \R\to \R$ such that $g(W_1)\sim \mu$ and thus $Y_t:= \mathbb{E}[g(W_1)|\mathcal{F}_t]$, $t\in [0,1]$, gives a martingale with $Y_0=0$ and $Y_1\sim \mu$. Based on the strong Markov property, the knowledge of the marginal distributions of the Brownian motion and It\^o's formula, Bass showed that there exists a function $\sigma \colon [0,1]\times \R\to [0,\infty)$ such that $Y$ is a weak solution to the stochastic differential equation 
\begin{equation}\label{eq:sde intro}
  \d Z_t = \sqrt{\sigma (t,Z_t)}\dd W_t \quad \text{with}\quad Z_0\sim \delta_0\quad \text{and} \quad Z_1 \sim \mu.
\end{equation}

\textbf{Step~2:} A time-change $\tau$ is constructed as the unique solution of the random ordinary differential equation 
\begin{equation*}
  \tau^\prime (t) =  \sigma (t,W_{\tau(t)})\quad \text{with}\quad\tau(0)=0
\end{equation*}
and it is shown that $(\tau (t))_{t\in [0,1]}$ is a family of stopping times with respect to the filtration generated by $W$. By general time-change arguments (or Dubins-Schwarz Theorem in the Brownian case), $(W_{\tau(t)})_{t\in [0,1]}$ turns out to be a weak solution of \eqref{eq:sde intro} and thus $(W_{\tau(t)})_{t\in [0,1]}$ has the same distribution as $(Y_t)_{t\in [0,1]}$ if $\sigma$ has enough regularity to ensure weak uniqueness to the stochastic differential equation~\eqref{eq:sde intro}. As $Y_1\sim \mu$ by Step~ 1, $\tau(1)$ is a solution to the classical Skorokhod embedding problem for Brownian motion.\smallskip

From an analytical point of view, Bass solved in Step 1 an inverse problem for a second order partial differential equation. Indeed, given the two marginal distributions $\delta_0$ and $\mu$, Bass first writes down a process $Y$ with some marginal distributions  $p(t,\cdot)$, $t\in [0,1]$, that match the prescribed marginals at times $0$ and $1$. He then finds a $\sigma$ such that the Fokker-Planck equation for the time-inhomogeneous generator $\frac{\sigma}{2} \frac{\partial^2}{\partial x^2}$ 
\begin{align}\label{dd}
  \int_\mathbb{R} f(x)\, p(t,\d x) - \int_\mathbb{R} f(x) \,\delta_0(\d x) =\int_0^t \int_\mathbb{R} \sigma(s,x) \frac{1}{2} \frac{\partial^2}{\partial x^2}f(x)\, p(s,\d x)\dd s ,\quad t\in [0,1],
\end{align}
is solved by the family of marginal distributions. To circumvent the probabilistic derivation of $\sigma$ through the clever choice of $Y$ in Step 1 of Bass (which does not extend to L\'evy processes) we directly solve the Fokker-Planck equation \eqref{dd} for $\sigma$ with a carefully chosen Ansatz for the marginals $p(t,\cdot)$, $t\in [0,1]$. A priori, our Ansatz is not related to a stochastic process $Y$ and we need to work hard to justify the existence of a process $Y$ with the prescribed marginals. For the special case of a Brownian motion, our marginals differ from Bass' marginals so also $\sigma$ differs and, as a consequence, our stopping time $\tau$ differs from Bass' stopping time.

\begin{remark} 
  The idea we implement stems from implied volatility theory and goes back to Dupire~\cite{Dupire1994}. 
  The original idea goes as follows: Suppose that in a Brownian market model $\d S_t=\sigma(t,S_t)\dd B_t$ all European call prices
  \begin{align*}
    C(T,K)=\E[\max\{S_T-K,0\}]=\int_\R \max\{x-K,0\} \phi(T,x) \dd x
  \end{align*}	
  at time $0$ for strike $K$ and maturity $T$ are known but $\sigma$ is only known to exist but not explicit. Here, $\phi(t,\cdot)$ denotes the marginal density of $S$ at time $t$. Differentiating the known call prices~$C$ twice with respect to the strike prices gives the key formula $\phi(T,K)=C_{KK}(T,K)$. This shows that from the knowledge of all option prices one can deduce the entire solution~$\phi$ to the Fokker-Planck equation. In order to identify the model that implied the observed option prices (i.e. recover $\sigma$ from the Fokker-Planck equation) Dupire suggested a formula for $\sigma$ that we recall below. A generalization of this idea was carried out in \cite{Carr2004} for jump diffusion models.
\end{remark}

\begin{remark}
  Similar ideas have been used e.g. in \cite{Hirsch2011}, \cite{Ekstrom2013} and \cite{Cox2011} to construct (martingale) diffusions that match prescribed marginal distributions at given (random) times. This provides alternative constructions of $Y$ in Step 1. By repeating Step 2 such processes could potentially be used to construct alternative solutions to the classical Skorokhod embedding problem for Brownian motion. Note that as in the case of Bass~\cite{Bass1983} these constructions are specific to the case of Brownian motion and, as we are interested in general L\'evy processes, we had to develop new ideas.
\end{remark}

Here, and in what follows, we say $\sigma$ is the (possibly time-dependent) local speed function of a process $Y$, if $Y$ has generator $\sigma \mathcal A$ for a time-homogeneous generator $\mathcal A$ of another (given) Markov process. The corresponding Fokker-Planck equation satisfied by the densities $\phi$ (if they exist) of $Y$ is
\begin{align}\label{bbb}
  \int_\mathbb{R} f(x)\, \phi(t,x)\dd x- \int_\mathbb{R} f(x) \,\mu_0(\d x) = \int_0^t \int_\mathbb{R} \sigma(s,x) \mathcal{A}f(x) \phi(s,x)\dd x \dd s,\quad t\geq 0,
\end{align}
for the starting distribution $\mu_0$.\smallskip

Our approach to the Skorokhod embedding problem takes up the ideas from implied volatility but is fundamentally different from option pricing at the same time: Dupire assumes a priori that the solution to the Fokker-Planck equation (equivalently all option prices) is implied by some~$\sigma$ and then recovers $\sigma$ from the option prices through the Fokker-Planck equation. In particular, the Fokker-Planck equation is a priori assumed to be well-posed. We proceed differently: we suggest a family $(\phi(t,\cdot))_{t\in [0,1]}$ of densities and hope to find a $\sigma$ so that the corresponding Fokker-Planck equation has the unique solution $\phi$. Of course, a priori there is no reason to believe this $\sigma$ exists and it strongly depends on assumptions posed upon $\phi$! Comparing with Bass' approach, instead of deriving $\phi$ as densities of $Y_t:= \mathbb{E}[g(W_1)|\mathcal{F}_t]$ we write down an Ansatz for $\phi$:
\begin{align}\label{eq:ph}
  \phi(t,\cdot)=th_1+(1-t)h_0,\quad t\in [0,1],
\end{align}
where $h_0,h_1$ are the densities from the Skorokhod embedding problem. The Ansatz for $\phi$ looks arbitrary (it is arbitrary, indeed) but below it turns out to work very nicely.
In order to derive a formula for $\sigma$ in terms of $\phi$, we follow the idea of Dupire: Assuming such a $\sigma$ exists, the transition densities $\phi$ have to solve the Fokker-Planck equation~\eqref{bbb} for nice test functions $f$. Taking derivatives with respect to~$t$ yields
\begin{align*}
  \int_\mathbb{R} f(x) \partial_t \phi(t,x) \dd x  = \int_{\mathbb{R}} \sigma(x,t)\mathcal{A}f(x)\phi(t,x) \dd x
  = \int_{\mathbb{R}} f(x) \mathcal{A}^*[\sigma(t,\cdot)\phi(t,\cdot)](x) \dd x,
\end{align*}
where $\mathcal{A}^*$ is the adjoint operator. Hence, $\sigma$ needs to satisfy the equation 
\begin{equation*} 
  \partial_t \phi(t,x) = \mathcal{A}^*[\sigma(t,\cdot)\phi(t,\cdot)](x)
\end{equation*}
which, solving for $\sigma$, gives the "generalized Dupire formula"
\begin{equation*}
  \sigma(t,x) = \frac{\left((\mathcal{A}^*)^{-1} \partial_t \phi(t,\cdot) \right)(x)}{\phi(t,x)} .
\end{equation*}
Plugging-in the choice \eqref{eq:ph} of $\phi$ and then using the assumption that there is a solution~$H$ to the Poisson equation $\mathcal A^*H=h_1-h_0$ yields the formula
\begin{equation}\label{eq:sigma}
  \sigma(t,x)=\frac{H(x)}{\phi(t,x)}.
\end{equation}
If now for this $\sigma$ there is a unique Markov process $Y$ with generator $\sigma \mathcal A$, then this is the analogue to the solution to \eqref{eq:sde intro} in the approach of Bass.\smallskip

Finally, we obtain from $Y$ the stopping time $\tau$ in the same way Bass did in his Step 2: Solve the random ODE $\tau'(t)=\sigma(t,Y_t)$ and define $\tau:=\tau(1)$. This is where the positivity assumption on $H$ enters the proof: a time-change should be an increasing function. Since by construction $Y_1$ has marginal distribution $\phi(1,x)\dd x=h_1(x)\dd x$ and $L_{\tau(1)}\sim Y_1$, $\tau$ is a solution to the Skorokhod embedding problem.\smallskip
  
It only remains to show that $\tau$ has finite mean, but this is immediate from the definitions above (integrating out the law $Y_t\sim \phi(t,x)\dd x$) and the assumed integrability of $H$:
\begin{align*}
  \mathbb{E}[\tau(1)] & = \mathbb{E}\bigg[\int_0^1 \sigma(t,L_{\tau (t)}) \dd t \bigg ]
  = \mathbb{E}\bigg[\int_0^1 \sigma(t,Y_t) \dd t \bigg ]
  = \int_0^1 \int_\mathbb{R} \sigma(t,x) \phi(t,x) \dd x \dd t   = \int_\mathbb{R} H(x) \dd x.
\end{align*}
\smallskip

Here is a summary of our strategy:
\begin{align*}
  \underbrace{\mu_0, \mu_1 \xrightarrow{\text{Ansatz}} \phi \xrightarrow{\text{Dupire formula}}\sigma \xrightarrow{\text{martingale problem}}\sigma \mathcal A\text{ Markov}}_{\text{replacing Step 1 of Bass}}\underbrace{\text{ process }Y\xrightarrow{\text{time-change}}\tau=\tau(1)}_{\text{extending Step 2 of Bass}}
\end{align*}

\begin{remark}
  The arguments rely on classical time-change techniques for Markov processes and existence/uniqueness results for Fokker-Planck equations with time-inhomogeneous coefficients. In the time-homogeneous case many of the results needed can be found in \cite{Ethier1986a}. 
  Results under minimal conditions for the time-inhomogeneous case are developed in the accompanying article~\cite{Nr2}. For convenience of the reader those results from~\cite{Nr2} which are required in the proof here are stated in the next section.
\end{remark}

\section{Preliminaries}

This section starts by stating the notation and definitions used throughout the article. In the subsequent subsections we then introduce the necessary preliminaries about L\'evy processes, time-changes and the associated Fokker-Planck equations.

\subsection{Notation and definitions}

$C_0(\mathbb{R})$ denotes the space of all continuous functions $f\colon \R \to \R$ satisfying $\lim_{|x|\to \infty } f(x) = 0$  and $C_b(\R)$ is the space of bounded continuous functions on $\R$. For $n\in \mathbb{N}$ let $C^n_0(\R)$ be the subset of functions $f\in C_0(\R)$ such that $f$ is $n$-times differentiable and all derivatives of order less or equal to $n$ belong to $C_0(\R)$ and we set $C_0^\infty(\R):=\bigcap_{n\in \mathbb{N}}C^n_0(\R)$. The spaces of functions with compact support $C_c(\R)$, $C_c^n(\R)$ and $C_c^\infty(\R)$ are defined analogously. The space $D_{\R}[0,\infty)$ stands for all maps $\omega \colon [0,\infty) \to \R$ which are right-continuous and have a left-limit at each point $t\in [0,\infty)$ (short: RCLL paths).  For $x \in \R$ and $\varepsilon > 0$, set $B_\varepsilon(x):= \{y \in \R \, :\, |x-y| < \varepsilon\}$. $\mathcal{P}(\R)$ (resp. $\mathcal{P}(D_\R[0,\infty))$) denotes the set of probability measures on $(\R,\mathcal{B}(\R))$ (resp. on $(D_\R[0,\infty),\mathcal{B}(D_\R[0,\infty))$). $B(\R)$ denotes the space of real-valued, bounded, measurable functions on $\R$ and $\|\cdot\|$ is the sup-norm. For $f \in C_0(\R)$ and a sequence $(f_n)_{n \in \mathbb{N}}$ with $f_n \in C_0(\R)$ we say that $f_n$ converges to $f$ in $C_0(\R)$ (and write $f_n \to f$ in $C_0(\R)$, etc.), if $\lim_{n \to \infty} \|f_n-f\|=0$, i.e.\ $f_n$ converges to $f$ uniformly. 
\smallskip

Moreover, we denote by $Z$ the coordinate process on $D_{\R}[0,\infty)$ and make the following definition (analogous to  \cite[Chap.~6, Thm.~1.1]{Ethier1986a}):
\begin{enumerate}[label= D.\arabic*]
  \item\label{def:regular}
        A measurable map $H \colon \R \to [0,\infty)$  is called \textit{regular} for $P\in \mathcal{P}(D_\R[0,\infty))$ if $P$-a.s.
        \begin{equation*}
          \inf \left\lbrace s \in [0,\infty)\,:\,\int_0^s H(Z_u)^{-1} \dd u =\infty \right\rbrace = \rho \quad \text{ and } \quad H(Z_ \rho)=0 \text{ on } \{\rho < \infty\}
        \end{equation*} 
        where $\rho := \inf \left\lbrace s \in [0,\infty)\,:\, H(Z_s) = 0 \right\rbrace$.
\end{enumerate}

Let $\mathcal{D} \subset C_b(\R)$, $\mathcal{A}\colon \mathcal{D} \to C_b(\R)$ linear and $\mu \in \mathcal{P}(\R)$. A \textit{solution to the RCLL-martingale problem for $(\mathcal{A},\mu)$} is an $\R$-valued process $(X_t)_{t \geq 0}$ with RCLL sample paths defined on some probability space $(\tilde{\Omega},\tilde{\mathcal{F}},\tilde{\mathbb{P}})$ such that for each $h \in \mathcal{D}$, the process 
\begin{equation*}
  h(X_t) - h(X_0) - \int_0^t \mathcal{A} h (X_s) \dd s,  \quad t \geq 0,
\end{equation*}
is an $(\mathcal{F}^X_t)_{t\geq 0}$-martingale and $\tilde{\mathbb{P}} \circ X_0^{-1} = \mu$, where $(\mathcal{F}^X_t)_{t\geq 0}$ denotes the filtration generated by~$X$.  The RCLL-martingale problem for $(\mathcal{A},\mu)$ is said to be \textit{well-posed} if there exists a solution and uniqueness holds, that is, if any two solutions $X$ and $\tilde{X}$ to the RCLL-martingale problem for $(\mathcal{A},\mu)$ have the same finite-dimensional distributions.

\subsection{L\'evy processes}\label{subsec:levyprelim}

Recall from the introduction that $(\alpha^2, \gamma, \nu)$ denotes the L\'evy triplet and 
\begin{align*} 
  \eta(u) = - \frac{1}{2}\alpha^2 u^2 + i u \gamma + \int_{\mathbb{R}\setminus \{0\}} (e^{i u y} - 1 - i u y \mathbbm{1}_{\{|y|\leq 1\}} )\,\nu(\d y),  \quad u \in \mathbb{R},
\end{align*}
is the characteristic exponent, i.e. $\mathbb{E}[e^{i u L_t}] = e^{t \eta(u)}$ for $u \in \mathbb{R}$ and $t \geq 0$. In what follows we collect the machinery that we need to study the Poisson equation for L\'evy processes in the next section.\smallskip

For $t \geq 0$ and $f \in C_0(\mathbb{R})$ define the transition semigroup $P_t f(x):= \E[f(L_t + x)], x \in \mathbb{R}$, and, for $q > 0$, $f \in C_0(\mathbb{R})$ the resolvent operators
\begin{equation*}
  U^q f(x) := \int_0^\infty e^{-q t} P_t f(x) \dd t ,  \quad x \in \mathbb{R}.
\end{equation*}
By dominated convergence, $f \in C_0(\mathbb{R})$ implies $P_t f \in C_0(\mathbb{R})$ for any $t \geq 0$ and thus 
\begin{equation*}
  D(\mathcal{A}) : = \big\{ f \in C_0(\mathbb{R}) \,:\, \lim_{t \to 0} t^{-1} ( P_t f - f ) \text{ exists in } C_0(\mathbb{R})\big \}
\end{equation*}
is well-defined. For $f \in D(\mathcal{A})$ define $\mathcal{A} f :=  \lim_{t \to 0} t^{-1} ( P_t f - f )$. Then, see \cite[Thm.~31.5]{Sato1999}, the generator $\mathcal{A}\colon D(\mathcal{A})  \to C_0(\mathbb{R})$ is linear, $C_0^2(\mathbb{R}) \subset D(\mathcal{A})$ and for $u \in C_0^2(\mathbb{R})$ it holds that
\begin{equation}\label{eq:LevyGenerator} 
  \mathcal{A} u (x) = \frac{1}{2} \alpha^2 u''(x) + \gamma u'(x) + \int_{\mathbb{R}\setminus\{0\}}[u(x+y)-u(x)-y u'(x) \mathbbm{1}_{\{|y|\leq 1\}}]\, \nu(\d y), \quad x \in \mathbb{R}.
\end{equation}
Furthermore, $C_c^\infty(\mathbb{R})$ is a core for $\mathcal{A}$. This means by definition that for any $f \in D(\mathcal{A})$, there exists a sequence $\{f_n\}_{n \in \mathbb{N}} \subset C_c^\infty(\mathbb{R})$ such that $\lim_{n \to \infty} f_n = f$ and $\lim_{n \to \infty} \mathcal{A}f_n = \mathcal{A} f$ in $C_0(\mathbb{R})$. In particular, we note the following:

\begin{lemma}\label{lem:levyFeller} 
  Suppose $L$ is a L\'evy process with L\'evy triplet $(\alpha^2,\gamma,\nu)$ and initial distribution $\mu_0$. Set $\mathcal{D}:=C_c^\infty(\mathbb{R})$ and  define $\mathcal{A}\colon\mathcal{D} \to C_0(\mathbb{R})$ as \eqref{eq:LevyGenerator} for $u \in \mathcal{D}$. Then the following hold:
  \begin{enumerate}
     \item[(i)] $\mathcal{D}$ is dense in $C_0(\R)$ and an algebra in $C_0(\R)$,
     \item[(ii)] there exists $\{\phi_n\}_{n \in \mathbb{N}} \subset \mathcal{D}$  such that $\sup_n \|\phi_n\| < \infty$, $\sup_n \|\mathcal{A} \phi_n\| < \infty$ and
     \[ \lim_{n \to \infty} \phi(x) = 1 \text{ and } \lim_{n \to \infty} \mathcal{A} \phi_n(x) = 0 \text{ for } x \in \R, \] 
     \item[(iii)] for any $\mu \in \mathcal{P}(\R)$, the RCLL-martingale problem for $(\mathcal{A},\mu)$ is well-posed,
          \item[(iv)] $L$ is a solution to the RCLL-martingale problem for $(\mathcal{A},\mu_0)$.
   \end{enumerate} 
  Note that the properties (i)-(iii) imply that Assumption~2.7 in \cite{Nr2} is satisfied.
\end{lemma}

\begin{proof} 
  This follows by general theory (as explained in Example~2.8 of \cite{Nr2}) since the transition semigroup is a positive, strongly continuous contraction semigroup on $C_0(\mathbb{R})$ by \cite[Thm.~31.5]{Sato1999} and since $\mathcal{D}$ is a core (see above) for the infinitesimal generator of $(P_t)_{t \geq 0}$. One could also verify (ii) by hand by taking $\phi \in C_c^\infty(\mathbb{R})$ with $\phi(x) = 1$ for $x \in [-1,1]$ and $\phi(x) = 0$ for $x \notin [-2,2]$ and setting $\phi_n(x) := \phi(x/n)$ for $x \in \mathbb{R}$, $n \in \mathbb{N}$.
\end{proof}

Two objects which are less popular in the study of L\'evy processes, but central for our purposes, are the potential operator and the adjoint, both of which we introduce next. From \cite[Remark~31.10]{Sato1999} or \cite[Thm.~4.1]{Sato1972} it follows that the L\'evy process admits a potential operator. By definition\footnote{More precisely, if $(P_t)_{t \geq 0}$ admits a potential operator, then $V=-\mathcal{A}^{-1}$ and $D(V)$ is dense in $C_0(\mathbb{R})$ by \cite[Thm.~2.3]{Sato1972} and \eqref{eq:potentialDomainChar} holds by definition of $V$ (c.f.\ equation (1.3) in \cite{Sato1972}). In \cite[Thm.~4.1]{Sato1972} it is proved that the semigroup $(P_t)_{t \geq 0}$ associated to $L$ indeed admits a potential operator.}  this means that $\mathcal{A}$ is injective, the domain $D(V):=\{\mathcal{A} f \, : \, f \in D(\mathcal{A})\}$ of the potential operator $V := - \mathcal{A}^{-1}$ is dense in $C_0(\mathbb{R})$ and, for $f, g \in C_0(\mathbb{R})$, 
\begin{equation}\label{eq:potentialDomainChar}
  g \in D(V) \text{ and } V g = f \quad \Longleftrightarrow \quad U^q g \to f \text{ in } C_0(\mathbb{R}) \text{ as } q \to 0.   
\end{equation}

Furthermore, for $t \geq 0$, set $L^*_t := - L_t$. Then $L^*$ is also a L\'evy process, called the dual L\'evy process, and its L\'evy triplet is given by $(\alpha^2,-\gamma,\nu^*)$ where $\nu^*(A):=\nu(\{-x \, : \, x \in A\})$ for $A \in \mathcal{B}(\mathbb{R})$. In other words, the characteristic exponent $\eta^*$ of $L^*$ is given for $u \in \mathbb{R}$ as $\eta^*(u) = \eta(-u)$ where $\eta$ is the characteristic exponent~\eqref{eq:symbol} of $L$. Since $L^*$ is also a L\'evy process, the transition semigroup, resolvent operator,  infinitesimal generator and potential operator have been defined above. We will denote them by $P_t^*$, $(U^q)^*$, $\mathcal{A}^*$ and $V^*$, respectively. \smallskip

For example, we denote by $\mathcal{A}^*$ the infinitesimal generator associated to the dual L\'evy process $L^*$ and refer to it as the dual of $\mathcal{A}$. Recall from the above that $C_0^2(\mathbb{R}) \subset D(\mathcal{A}^*)$ and for $u \in C_0^2(\mathbb{R})$, 
\begin{equation}\label{eq:adjointDef}
  \mathcal{A}^* u (x) = \frac{1}{2} \alpha^2 u''(x) - \gamma u'(x) + \int_{\mathbb{R}\setminus\{0\}}[u(x-y)-u(x)+y u'(x) \mathbbm{1}_{\{|y|\leq 1\}}]\, \nu(\d y), \quad x \in \mathbb{R}.
\end{equation}

\begin{remark}
  The semigroup $(P_t)_{t \geq 0}$ and the operator $\mathcal{A}$ are defined on (a subset of) $C_0(\mathbb{R})$ in the present context. We define $\mathcal{A}^*$ also on $C_0(\mathbb{R})$ (and not on the dual space of $C_0(\mathbb{R})$ as in \cite{Sato1972}). The next lemma justifies the $*$-notation. 
\end{remark}

The following lemma is immediate, it identifies the dual generator $\mathcal A^*$ as the adjoint operator of $\mathcal A$. For the proof of Theorem~\ref{thm:main} we will not need all the cases, but we have included the other ones for completeness.

\begin{lemma}\label{lem:adjoint}
  Suppose $L$ is a L\'evy process with L\'evy triplet $(\alpha^2,\gamma,\nu)$ and denote by  $\mathcal{A}$ its infinitesimal generator $\mathcal{A}\colon D(\mathcal{A}) \to C_0(\mathbb{R})$ and $\mathcal{A}^*\colon D(\mathcal{A}^*) \to C_0(\mathbb{R})$ the infinitesimal generator of $-L$. Then
  \begin{equation}\label{eq:IntegralAdjoint} 
    \int_\mathbb{R} \mathcal{A}f(x)g(x) \dd x = \int_\mathbb{R} f(x) \mathcal{A}^* g (x) \dd x,
  \end{equation}
  for any $f \in D(\mathcal{A})$, $g \in D(\mathcal{A}^*)$ such that either $f, g \in L^1(\mathbb{R})$ or $f, \mathcal{A}f \in L^1(\mathbb{R})$ or $g, \mathcal{A}^*g \in L^1(\mathbb{R})$ or $\mathcal{A}f, \mathcal{A}^*g \in L^1(\mathbb{R})$.
\end{lemma}

\begin{proof} 
  \textbf{Case 1}, $f, g \in L^1(\mathbb{R})$:  For $t \geq 0$ and $x \in \mathbb{R}$ denote by $P^*$ the transition semigroup of $-L$, i.e. $P^*_t g(x)= \E[g(x-L_t)]$. By Fubini's Theorem $g \in L^1(\mathbb{R})$ implies $P^*_t g \in L^1(\mathbb{R})$ for any $t \geq 0$. By \cite[Chap.~II, Prop.~1]{Bertoin1996} for any $t \geq 0$ it holds that
  \begin{equation}\label{eq:bertoinAdjoint}
    \int_\mathbb{R} P_t f(x) g(x) \dd x = \int_\mathbb{R} f(x) P_t^* g(x) \dd x.
  \end{equation}

  To be precise, in \cite[Chap.~II, Prop.~1]{Bertoin1996} $f$ and $ g$ are assumed non-negative, but by considering positive and negative parts separately and using $g \in L^1(\mathbb{R})$ and $P_t^*g \in L^1(\mathbb{R})$, \cite[Chap.~II, Prop.~1]{Bertoin1996} implies \eqref{eq:bertoinAdjoint}.

  Using the definition of $\mathcal{A}$, $f \in D(\mathcal{A})$ and $g \in L^1(\mathbb{R})$ to apply dominated convergence in the first step and $g \in D(\mathcal{A}^*)$ and $f \in L^1(\mathbb{R})$ in the last step, one obtains
  \begin{align*}
    \int_\mathbb{R} \mathcal{A}f(x)g(x) \dd x 
    & = \lim_{t \to 0} \frac{1}{t} \int_\mathbb{R} (P_t f(x) - f(x)) g(x) \dd x \\ 						&\stackrel{\mathmakebox[\widthof{=}]{\eqref{eq:bertoinAdjoint}}}{=} 
    \lim_{t \to 0} \frac{1}{t} \int_\mathbb{R} f(x) (P_t^* g(x) - g(x)) \dd x \\ 
     & = \int_\mathbb{R} f(x) \mathcal{A}^* g (x) \dd x
  \end{align*}
  and so \eqref{eq:IntegralAdjoint} has been established under the assumption $f, g \in L^1(\mathbb{R})$.\smallskip
  
  \textbf{Case 2},  $f, \mathcal{A}f \in L^1(\mathbb{R})$ or $g, \mathcal{A}^*g \in L^1(\mathbb{R})$: 
  Suppose $g, \mathcal{A}^*g \in L^1(\mathbb{R})$, the other case can be treated by the same argument. Since $C_c^\infty(\mathbb{R})$ is a core for $\mathcal{A}$, there exists  $\{f_n\}_{n \in \mathbb{N}} \subset C_c^\infty(\mathbb{R})$ with $\lim_{n \to \infty} f_n = f$ and $\lim_{n \to \infty} \mathcal{A}f_n = \mathcal{A} f$ in $C_0(\mathbb{R})$, i.e. uniformly. But $f_n, g \in L^1(\mathbb{R})$ and so \eqref{eq:IntegralAdjoint} holds for $f_n$ and $g$ for any $n \in \mathbb{N}$. Furthermore, the assumptions $g, \mathcal{A}^*g \in L^1(\mathbb{R})$ allow us to apply dominated convergence and so \eqref{eq:IntegralAdjoint} also holds for $f$ and $g$, as desired.\smallskip
  
  \textbf{Case 3}, $\mathcal{A}f$, $\mathcal{A}^*g \in L^1(\mathbb{R})$:
  For the proof of the last part, denote by $V$ and $V^*$ the potential operators associated to $\mathcal{A}$ and $\mathcal{A}^*$, respectively. We claim that for any $\tilde{f} \in D(V)$, $\tilde{g} \in D(V^*)$ with $\tilde{f}, \tilde{g} \in L^1(\mathbb{R})$ it holds that
  \begin{equation}\label{eq:adjointPotential}
    \int_\mathbb{R} V \tilde{f}(x) \tilde{g}(x) \dd x = \int_\mathbb{R} \tilde{f}(x) V^* \tilde{g} (x) \dd x. 
  \end{equation}
  Once this is established, we may set $\tilde{f}:=\mathcal{A}f$, $\tilde{g} := \mathcal{A}^*g$ and apply \eqref{eq:adjointPotential} to deduce \eqref{eq:IntegralAdjoint}.

  So assume $\tilde{f} \in D(V)$, $\tilde{g} \in D(V^*)$ and $\tilde{f}, \tilde{g} \in L^1(\mathbb{R})$.   
  For $q > 0$ and $x \in \mathbb{R}$ denote by $(U^q)^*$ the resolvent operator of $-L$, i.e. $(U^q)^* \tilde{g} (x)= \int_0^\infty e^{-q t} P^*_t \tilde{g} (x) \dd t$ where $P^*$ is the transition semigroup of $-L$ as above.
  By Fubini's Theorem, $\tilde{g} \in L^1(\mathbb{R})$ implies $(U^q)^* \tilde{g}  \in L^1(\mathbb{R})$ for any $q > 0$. 
  By \cite[Chap.~II, Prop.~1]{Bertoin1996} for any $q > 0$ it holds that
  \begin{equation}\label{eq:bertoinAdjointResolvent}
    \int_\mathbb{R} U^q \tilde{f}(x) \tilde{g} (x) \dd x = \int_\mathbb{R} \tilde{f}(x) (U^q)^* \tilde{g} (x) \dd x.
  \end{equation}
  
  By the same argument as in Case 2, \cite[Chap.~II, Prop.~1]{Bertoin1996} implies \eqref{eq:bertoinAdjointResolvent}.
      
  Using \eqref{eq:potentialDomainChar}, $\tilde{f} \in D(V)$ and $\tilde{g} \in L^1(\mathbb{R})$ one may let $q \to 0$ and apply dominated convergence to obtain that the left-hand side of \eqref{eq:bertoinAdjointResolvent} converges to the left-hand side of \eqref{eq:adjointPotential} and analogously for the right-hand side. Thus \eqref{eq:adjointPotential} is indeed established. 
\end{proof}

\subsection{Fokker-Planck equations with time-dependent coefficients}\label{sec:Fokker-Planck}

The proof of Theorem~\ref{thm:main} relies on time-change arguments and uniqueness results for time-inhomogeneous Fokker-Planck equations as developed in the accompanying paper~\cite{Nr2}. To facilitate the reading of the present article we now collect the results that we need, but we refer the reader to \cite{Nr2} for the proofs.

Recall that $\mathcal A$ is the generator of a one-dimensional L\'evy process $L$ defined on a probability space~$(\Omega,\mathcal{F},\P^{\mu_0})$ and $L$ is adapted to the complete and right-continuous filtration~$(\mathcal{F}_t)$. As outlined in Section~\ref{sec:sketch}, we aim at constructing a time-change~$\delta$ characterized by the random Carath\'eodory differential equation 
\begin{equation}\label{eq:ode time-change}
  \delta(s) = \int_0^s \sigma(r,L_{\delta(r)}) \dd r, \quad s\in [0,\thor],
\end{equation}
such that the marginal distributions of the time-changed process $X_s:=L_{\delta(s)}$ are given by the linear interpolation~\eqref{eq:ph}. We will write $\sigma$ in the form $\sigma(t,x) = H(x) / \phi(t,x) = H(x)\tilde{\sigma}(t,x) $ as in~\eqref{eq:sigma}, where $H$ is the solution to the Poisson equation~\eqref{eq:poisson eq}, $\phi$ is defined as in~\eqref{eq:ph} and $\tilde \sigma(t,x):=1/\phi(t,x)$. In this factorisation form we can distinguish assumptions on the L\'evy process (captured in $H$) and assumptions on the densities $h_0, h_1$ (captured in $\tilde \sigma$).\smallskip

In the following assumptions on the involved objects are formulated, which guarantee the solvability of equation~\eqref{eq:ode time-change} and which are sufficient to prove existence and uniqueness results for the Fokker-Planck equation associated to the operator $\sigma \mathcal{A}$. Lemma~\ref{lem:levyFeller} above shows that $L$, $\mathcal{D}:=C_c^\infty(\mathbb{R})$ and $\mathcal A$ defined by \eqref{eq:LevyGenerator} indeed fall within the framework of \cite{Nr2}. We divide the assumptions in such a way that allows to distinguish as good as possible between assumptions on the L\'evy process~$L$ and the ingredients coming form the densities~$h_0$, $h_1$:

\begin{enumerate}[label= A.\arabic*]
  \item\label{ass:sigma} \textbf{Regularity of \boldmath$\sigma$:}
        Let $\sigma \colon [0,\infty) \times \R \to [0,\infty)$   be of the form $\sigma (t,x):= H(x) \tilde{\sigma}(t,x)$ for $(t,x)\in[0,\infty) \times \R$ with $\tilde{\sigma}(t,x) \equiv 0$ for $t> \thor$ and such that
        \begin{enumerate}
          \item[(i)] $H \colon \R \to [0,\infty)$ is measurable,
          \item[(ii)] $\tilde{\sigma} \colon [0,\thor] \times \R \to (0,\infty)$ is measurable and satisfies the following: for each compact set  $K \subset \R$ there exists $C_1,C_2,C_3 >0$ such that 
          \begin{equation*}
            |\tilde{\sigma}(t,x) - \tilde{\sigma}(s,x)|  \leq C_1 |t - s|\quad \text{and}\quad C_2 \leq \tilde{\sigma}(t,x) \leq  C_3,
          \end{equation*}
          for all $s,t \in [0,1]$ and for all $x \in K$.
        \end{enumerate}
  \item\label{ass:SigmaBounded} \textbf{Boundedness of \boldmath$\sigma$:}
       $\sigma \colon [0,\infty) \times \R \to [0,\infty)$ is bounded.
  \item\label{ass:Hregular} \textbf{Regularity of \boldmath$H$:} 
        Let $H \colon \R \to [0,\infty)$ measurable. Assume that for any $x \in \R$, $H$ is regular for $P$ in the sense of Definition~\ref{def:regular}, where $P$ is the law on $D_{\R}[0,\infty)$ of $L$ under $\P^{\delta_x}$.
\end{enumerate}

\subsubsection{Time-inhomogeneous random time-changes}

The next proposition ensures that the random Carath\'eodory differential equation~\eqref{eq:ode time-change} indeed provides a suitable time-change~$\delta$ under regularity assumption of $\sigma$ and $H$. Moreover, it verifies that $(\delta(s))_{s\in[0,\thor]}$ are stopping times with respect to the filtration generated by the L\'evy process $L$.

\begin{proposition}\label{prop:timechange}
  Assume that 
  \begin{itemize}
    \item $\sigma = H \tilde{\sigma}$ is given as in Assumption~\ref{ass:sigma},
    \item $H$ is regular for $P$ in the sense of Definition~\ref{def:regular}, where $P$ the law on $D_{\R}[0,\infty)$ of $L$ under $\P^{\mu_0}$,
    \item Assumption~\ref{ass:SigmaBounded} holds.
  \end{itemize}
  Then the family of random times $(\delta(s))_{s \in [0,\thor]}$ given by
     \begin{equation}\label{eq:stoppingTime2} 
    \delta(s) :=  \inf\{t \in [0,\rho)\,:\, \Delta(t) \geq s\} \wedge \rho, \quad s\in [0,\thor],
  \end{equation} 
  is well-defined, where $\Delta$ is the unique solution to the Carath\'eodory differential equation 
  \begin{equation}\label{eq:Delta} 
    \Delta(t) = \int_0^t \sigma(\Delta(r),L_r)^{-1} \dd r,\quad t \in [0, \delta (\thor))
  \end{equation}
  and 
    \begin{equation}\label{eq:rho}
    \rho := \inf \left\lbrace s \in [0,\infty)\,:\,\int_0^s H(L_u)^{-1} \dd u =\infty \right\rbrace .
  \end{equation}
 Furthermore
  \begin{enumerate}
    \item[(i)] $\delta \colon [0,\thor] \to [0,\infty)$ is non-decreasing and absolutely continuous, $\P^{\mu_0}$-a.s.,  
    \item[(ii)] $\delta (\thor)$ is finite, $\P^{\mu_0}$-a.s., 
    \item[(iii)] $\delta$ solves the Carath\'eodory differential equation~\eqref{eq:ode time-change},
    \item[(iv)] For any $s \in [0,\thor]$, $\delta(s)$ is an $(\mathcal{F}_t)$-stopping time.
  \end{enumerate}
\end{proposition}

Additionally, we present a useful condition for verifying the regularity of $H$ (cf. Definition~\ref{def:regular}) as required in order to apply Proposition~\ref{prop:timechange}.

\begin{proposition}\label{prop:regular}
  Let $\mu_0 \in \mathcal{P}(\R)$, denote by $P$ the law on $D_{\R}[0,\infty)$ of $L$ under $\P^{\mu_0}$ and let $H \in D(\mathcal{A})$ with $H \geq 0$. Then $H$ is regular for $P$.
\end{proposition}

\subsubsection{Existence and uniqueness for the Fokker-Planck equation associated to $\sigma \mathcal{A}$}

The non-decreasing stopping times constructed in Proposition~\ref{prop:timechange} can be used to define the time-changed process $X_s:=L_{\delta(s)}$, $s \in [0,1]$. The next result shows that the marginal distributions of $X$ satisfy the Fokker-Planck equation associated to the operator $\sigma \mathcal{A}$.

\begin{proposition}\label{prop:density time changed}
  Let $\mu_0 \in \mathcal{P}(\R)$ and let $\sigma$, $\delta$ be given as in Proposition~\ref{prop:timechange}. For $s \geq 0$, denote by $p(s,\cdot)$ the law of $X_s = L_{\delta(s)}$, where $\delta(s):= \delta(\thor)$ for $s > \thor$. Then one has: 
  \begin{itemize}
    \item[(i)] for any $g \in B([0,\infty)\times \R)$ the function 
          \begin{equation}\label{eq:measurabilityKolmogorovEqn}
            s \mapsto \int_E g(s,x)\, p(s,\d x) \text{ is measurable.}
         \end{equation} 
    \item[(ii)] $(p(s,\d x))_{s \in [0,\thor]}$ satisfies the Fokker-Planck equation, i.e. for any $f \in C_c^\infty(\mathbb{R})$,
          \begin{equation}\label{eq:KolmogorovEqn}
            \int_{\R} f(x)\, p(t,\d x) - \int_{\R} f(x) \,\mu_0(\d x) = \int_0^t \int_{\R} \sigma(s,x) \mathcal{A}f(x)\, p(s,\d x) \dd s,\quad t \in [0,\thor].
          \end{equation}
  \end{itemize}
\end{proposition}

Finally, we provide a uniqueness result for the Fokker-Planck equation~\eqref{eq:KolmogorovEqn} associated to the operator $\sigma \mathcal{A}$. In other words, the sufficient conditions provided in Theorem~\ref{thm:KolmogorovUniqueness} below guarantee that the one-dimensional marginal laws of $X_\cdot= L_{\delta(\cdot)}$, (which satisfy \eqref{eq:KolmogorovEqn} by Proposition~\ref{prop:density time changed}) are the only family of probability measures that satisfy~\eqref{eq:KolmogorovEqn} for all $f \in C_c^\infty(\R)$.

\begin{theorem}\label{thm:KolmogorovUniqueness} 
  Suppose $\sigma= H \tilde{\sigma}$ satisfies Assumptions~\ref{ass:sigma}, \ref{ass:SigmaBounded} and~\ref{ass:Hregular}. Then uniqueness holds for \eqref{eq:KolmogorovEqn}: 
  If $(q(t,\cdot))_{0\leq t \leq \thor}$ and $(p(t,\cdot))_{0\leq t \leq \thor}$ are two families of probability measures on $\R$ which both satisfy \eqref{eq:measurabilityKolmogorovEqn} and \eqref{eq:KolmogorovEqn} for all $f \in C_c^\infty(\mathbb{R})$ and $q(0,\cdot) = \mu_0 = p(0,\cdot)$, then $q(s,\cdot) = p(s,\cdot)$ for all $s \in [0,\thor]$, where $\mu_0 \in \mathcal{P}(\R)$.
\end{theorem}

\section{Proofs}

In this section we prove Theorem~\ref{thm:main}. The exposition is structured as follows: Firstly, in Section~\ref{sec:potential} preliminary results on L\'evy processes and the associated Poisson equation are presented. Secondly, in Section~\ref{sec:necessity} it is proved that \eqref{BJ2} is indeed necessary for the existence of a finite mean Skorokhod embedding. Finally, in Section~\ref{sec:sufficiency} it is proved that $\tau$ in Theorem~\ref{thm:main} is a finite mean solution to the Skorokhod embedding problem (SEP), thereby also proving sufficiency of~\eqref{BJ2}.

\subsection{The Poisson Equation for L\'evy Processes}\label{sec:potential}

In this section we lay the foundation for the proof of Theorem~\ref{thm:main}. As sketched in Section~\ref{sec:sketch}, it is crucial to understand the solvability of the Poisson equation $\mathcal A^*H=h_1-h_0$ and properties of the solution $H$. As sketched below Remark~\ref{remark}, the Poisson equation for L\'evy processes can be tackled with Fourier transforms. Throughout this section we take $\mu_0:=\delta_0$ and set $\P:=\P^{\mu_0}$.

\subsubsection{Solving the Poisson Equation using the Fourier Transform}\label{poisson}

By definition of the potential operator $V^* = - (\mathcal{A}^*)^{-1}$ in the Subsection~\ref{subsec:levyprelim}, for $g \in D(V^*)$ the function $H = - V^* g $ is the unique solution to the Poisson equation
\begin{equation}\label{eq:ellipticEqLevy} 
  \mathcal{A}^* H =g
\end{equation}
in $C_0(\mathbb{R})$. In this section we study the solvability of \eqref{eq:ellipticEqLevy} and further properties of solutions. 
 \smallskip

The first proposition justifies the heuristic given in the introduction below Remark~\ref{remark} and, hence, the occurrence of the function $H$ in Theorem~\ref{thm:main}. Note that the appearing assumption $g\in L^1(\R)$ will not pose any restriction as in our applications $g=h_1-h_0$ and $h_0,h_1$ are probability densities.

\begin{proposition}\label{p1}
  If $g \in C_0(\mathbb{R}) \cap L^1(\mathbb{R})$ and $\xi \mapsto \frac{\hat{g}(\xi)}{\eta(\xi)} \in L^1(\mathbb{R})$, then there is a unique solution $H \in D(\mathcal A^*) \subset C_0(\R)$ to the Poisson equation $\mathcal A^*H=g$ and 
  \begin{equation}\label{eq:potentialFourierRepres}
    H(x)= \frac{1}{2 \pi} \int_\mathbb{R} \frac{\hat{g}(\xi)}{\eta(\xi)}e^{-i x \xi} \dd \xi , \quad x \in \mathbb{R}.
  \end{equation} 
\end{proposition}

\begin{proof}
  We start with some preliminary facts that do not use the assumption of the proposition. Note that $g \in C_0(\mathbb{R}) \cap L^1(\mathbb{R})$ implies $g \in L^\infty(\mathbb{R})\cap L^1(\mathbb{R})$. Hence, one can use Fubini's Theorem to see that for any $q>0$, $(U^q)^* g (x)= \int_0^\infty e^{-q t} P^*_t g (x) \dd t$ is in $L^1(\mathbb{R})$. Taking the Fourier transform we obtain (see e.g. \cite[Chap.~I, Prop.~9]{Bertoin1996}) for any $q > 0$
  \begin{equation}\label{eq:ResolventFourier}
    \widehat{(U^q)^* g}(\xi) = (q - \eta(\xi))^{-1} \hat{g}(\xi),  \quad \xi \in \mathbb{R}.
  \end{equation}
  By \eqref{eq:charFctLevy} it holds that $|e^{\eta(\xi)}| = |\E[e^{i \xi L_1}]| \leq 1$ and thus 
  \begin{equation}\label{eq:symbolNegRealPart}
    \mathrm{Re}(\eta(\xi)) \leq 0 \quad \text{ for any }  \xi \in \mathbb{R},
  \end{equation}
  where $\mathrm{Re}(z)$ denotes the real part of $z \in \mathbb{C}$. In particular, the right-hand side of \eqref{eq:ResolventFourier} is indeed well-defined and
  \begin{equation}\label{eq:symbolEstimate1}
    \forall \xi \in \mathbb{R}, q >0 \,:\, |\eta(\xi)| \leq |q - \eta(\xi)|.
  \end{equation}
  By assumption, the inverse Fourier transform of $\xi \mapsto \frac{\hat{g}(\xi)}{\eta(\xi)}$, given by $H$ in \eqref{eq:potentialFourierRepres}, is well-defined. Furthermore, by \eqref{eq:symbolEstimate1} and our  assumption, for any $q > 0$ the right-hand side of \eqref{eq:ResolventFourier} is integrable and so, as $(U^q)^* g \in C_0(\mathbb{R}) \cap L^1(\mathbb{R})$ and $\widehat{(U^q)^* g} \in L^1(\mathbb{R})$, by Fourier inversion and \eqref{eq:ResolventFourier}, $(U^q)^*g$ can be represented as 
  \begin{equation}\label{eq:auxEq25}
    (U^q)^* g(x) = \frac{1}{2\pi} \int_\mathbb{R} \frac{\hat{g}(\xi) e^{-i x \xi}}{q - \eta(\xi)} \dd \xi, \quad x \in \mathbb{R}.
  \end{equation}
  Thus, one has
  \[\begin{aligned} \sup_{x \in \mathbb{R}} |-H(x)-(U^q)^* g(x)| & = \sup_{x \in \mathbb{R}} \left| \frac{1}{2 \pi} \int_\mathbb{R} \hat{g}(\xi)e^{-i x \xi}\left(\frac{1}{\eta(\xi)} + \frac{1}{q-\eta(\xi)} \right) \dd \xi \right| \\ & \leq \frac{1}{2 \pi} \int_\mathbb{R} |\hat{g}(\xi)|\left|\frac{1}{\eta(\xi)} + \frac{1}{q-\eta(\xi)}  \right| \dd \xi.  \end{aligned} \]
  However, by \eqref{eq:symbolEstimate1} the last integrand is bounded from above by $\xi \mapsto 2 \hat{g}(\xi) / \eta(\xi)$, which is integrable by assumption, and so we can let $q \to 0$ and apply dominated convergence to conclude $(U^q)^* g \to -H$ as $q \to 0$ in $C_0(\mathbb{R})$. By \eqref{eq:potentialDomainChar} (for the dual L\'evy process $L^*$), this implies $g \in D(V^*) $ and $V^* g = -H$ or, in other words, $\mathcal A^*H=g$.
\end{proof}

In the next proposition we show that under the assumptions of Theorem~\ref{thm:main}, it holds that $H \in D(\mathcal{A}^*)\cap D(\mathcal{A})$. This property will be crucial to guarantee uniqueness for the time-change in the proof of sufficiency of Theorem~\ref{thm:main} and forces us to assume that the densities $h_0, h_1$ are ``sufficiently smooth'' in the next sections.

\begin{proposition}\label{prop:levyregularaux}
    Suppose $h_0$, $h_1$ are as in Theorem~\ref{thm:main} and \eqref{BJ2} holds with $H$ as in \eqref{b}. Then $H \in D(\mathcal{A}) \cap D(\mathcal{A}^*)$.
\end{proposition}

\begin{proof}
    Set $g:=h_1-h_0$.  By Proposition~\ref{p1}, $H \in D(\mathcal{A}^*)$ and so we only need to verify $H \in D(\mathcal{A})$. If $L$ is of type S, then it is symmetric. In particular $\mathcal{A}=\mathcal{A}^*$ and hence the claim. If $L$ is of type 0 or D, then (as established in the proof of Lemma~\ref{lemma} below), $\hat{g} \in L^1(\R)$ and we now show that this implies $H \in D(\mathcal{A})$. 
    
   Since the complex conjugate of $\eta(\xi)$ is given by $\eta(-\xi)$ for all $\xi \in \R$ and since $\hat{g} \in L^1(\R)$, the function 
  \[ 
    f(x):=  \frac{1}{2 \pi} \int_\mathbb{R} \frac{\hat{g}(\xi)\eta(-\xi)}{\eta(\xi)}e^{-i x \xi} \dd \xi , \quad x \in \mathbb{R}, 
  \]
  is well-defined and, by the Riemann-Lebesgue Theorem, $f \in C_0(\R)$. Inserting~\eqref{eq:potentialFourierRepres} in the definition, applying Fubini's Theorem and using~\eqref{eq:charFctLevy} and the definition of $f$ yields
  \begin{align*}
    \sup_{x \in \R} \left|\frac{1}{t}( P_t H(x) - H(x) ) - f(x) \right| & =  \sup_{x \in \R} \left|\frac{1}{t}(\E[H(L_t+x)] - H(x) ) - f(x) \right| \\
    & = \sup_{x \in \R} \left| \frac{1}{2 \pi t} \int_\mathbb{R} \frac{\hat{g}(\xi)}{\eta(\xi)}(\E[e^{- i  L_t \xi}]-1)e^{-i x \xi} \dd \xi - f(x) \right| \\
    & \leq \frac{1}{2 \pi} \int_\mathbb{R} \frac{|\hat{g}(\xi)|}{|\eta(\xi)|}\left|\frac{e^{ t \eta(-\xi)}-1}{t} - \eta(-\xi) \right| \dd \xi, 
  \end{align*} 
  which tends to $0$ as $t \downarrow 0$, by dominated convergence. By definition, this implies $H \in D(\mathcal{A})$ and $\mathcal{A} H = f$.
  To see that dominated convergence can be applied, recall \eqref{eq:symbolNegRealPart} and so for all $\xi \in \R$, $t > 0$,
  \[
    \frac{|\hat{g}(\xi)|}{|\eta(\xi)|}\left|\frac{e^{ t \eta(-\xi)}-1}{t} - \eta(-\xi) \right| \leq 2 |\hat{g}(\xi)|.
  \] 
\end{proof}
The rest of this section may be skipped on first reading, all following propositions are not needed for the proof of Theorem \ref{thm:main}.\smallskip

We give conditions on $g$ and $\eta$ so that Proposition \ref{p1} implies the existence of the solution $H$, conditions that imply $H\in L^1(\R)$ and that $H$ is Lipschitz continuous. For future applications, those might be useful to verify the conditions of Theorem \ref{thm:main}.

\begin{proposition}
  Assume the non-degeneracy condition $\eta(u) \neq 0$ for all $ u \neq 0$ (i.e. $L$ is non-lattice) and either $\nu \neq 0$ or $\alpha \neq 0$.
  If $g \in C_0(\mathbb{R}) \cap L^1(\mathbb{R})$, $x \mapsto x^i g(x) \in L^1(\mathbb{R})$ for $i = 1,2$,
  \begin{equation}\label{eq:momentsVanish}
    \int_\mathbb{R} g(x) x^j \dd x = 0 , \quad j =0,1,
  \end{equation}
  and there exists $R > 0$ such that 
  \begin{equation}\label{eq:integrableAtInfinityNew} 
    \int_{|\xi|>R} \frac{|\hat{g}(\xi)|}{|\eta(\xi)|} \dd \xi < \infty, 
  \end{equation}
  then $\xi\mapsto\frac{\hat{g}(\xi)}{\eta(\xi)} \in L^1(\R)$.
\end{proposition}

\begin{proof}
  First note that in the present one-dimensional setup, the law of $L_1$ is degenerate (in the sense of \cite[Def.~24.16]{Sato1999}) if and only if there exists $a \in \mathbb{R}$ with $L_1 = a$, $\P$-a.s.  Since we have assumed $\alpha \neq 0$ or $\nu \neq 0$, this is not the case here (see \cite[Thm.~24.3]{Sato1999}). In particular, we may apply \cite[Prop.~24.19]{Sato1999} and obtain that there exist $\varepsilon' > 0$ and $c > 0$ such that 
  \begin{equation}\label{eq:auxEq20}
    |\E[e^{i \xi L_1}]| \leq 1 - c |\xi|^2  \quad \text{ for } |\xi| < \varepsilon'.
  \end{equation}
  By \eqref{eq:charFctLevy}, the left-hand side of \eqref{eq:auxEq20} is greater or equal than $e^{\mathrm{Re}(\eta(\xi))}$ and thus there exist $C>0$ and $\varepsilon >0$ such that 
  \begin{equation}\label{eq:auxEq21}
    - \mathrm{Re} (\eta(\xi)) \geq - \log(1-c|\xi|^2) \geq C |\xi|^2 
  \end{equation}
  for all $\xi \in B_\varepsilon(0)$. 
  
  On the other hand, for $g \neq 0$ (if $g = 0$, then the claim trivially holds), we may decompose $g = g^+ - g^- $ with $g^+\geq 0$, $g^- \geq 0$. Setting $c_0 := \int_\mathbb{R} g^+(x) \dd x$, \eqref{eq:momentsVanish} with $i=0$ implies $c_0 = \int_\mathbb{R} g^-(x) \dd x$ and thus $c_0 > 0$. Setting $h_1:=g^+ /c_0$ and $h_0:= g^- / c_0$, both $h_0$ and $h_1$ are probability densities and so we may apply \cite[Prop.~2.5~(ix)]{Sato1999} to $h_0$ and $h_1$ to obtain that, by our moment assumptions, $\hat{g} \in C^2(\mathbb{R})$ and, by \eqref{eq:momentsVanish}, $\hat{g}(0) = \hat{g}'(0) = 0$. Taylor expanding around $0$, we therefore obtain
  \begin{equation}\label{eq:auxEq22}
    |\hat{g}(\xi)| \leq C_0 \xi^2 
  \end{equation}
  for some $C_0 > 0$ and all $\xi \in B_\varepsilon(0)$. Combining \eqref{eq:auxEq21} and \eqref{eq:auxEq22} yields
  \begin{equation*} 
    |\hat{g}(\xi)|  \leq C_0 \xi^2 \leq   - \mathrm{Re}(\eta(\xi)) \frac{C_0}{C} \leq \frac{C_0}{C} |\eta(\xi)| 
  \end{equation*}
  for all $\xi \in B_\varepsilon(0)$ and thus $\xi\mapsto\frac{\hat{g}(\xi)}{\eta(\xi)}$ is locally bounded at zero. Since $\eta(u) \neq 0$ for $u \neq 0$ and $\hat{g}$ and $\eta$ are continuous, it follows that $\xi\mapsto\frac{\hat{g}(\xi)}{\eta(\xi)}$ is bounded on any compact subset of $\R$. Combining this with \eqref{eq:integrableAtInfinityNew} yields the claim.
\end{proof}

\begin{proposition}
  Suppose $g$ and $H$ are as in Proposition~\ref{p1}. If in addition for some $R > 0$,
  \begin{equation}\label{eq:integrableAtInfinity}
    \int_{|\xi|>R} \frac{|\xi| |\hat{g}(\xi)|}{|\eta(\xi)|} \dd \xi < \infty,
  \end{equation}
  then $H$ is Lipschitz continuous.
\end{proposition}

\begin{proof}
  By assumption,
  \[  \int_{|\xi|\leq R} \frac{|\xi| |\hat{g}(\xi)|}{|\eta(\xi)|} \dd \xi \leq  R \int_\R \frac{|\hat{g}(\xi)|}{|\eta(\xi)|} \dd \xi < \infty\] 
  and combining this with \eqref{eq:integrableAtInfinity} yields 
  \begin{equation}\label{eq:auxEq24}
    L := \int_\mathbb{R} \frac{|\xi| |\hat{g}(\xi)|}{|\eta(\xi)|} \dd \xi < \infty.
  \end{equation}
  
  On the other hand, precisely as in the proof of Proposition~\ref{p1} we may apply Fourier inversion to write, for any $q > 0$, $(U^q)^*$ as \eqref{eq:auxEq25}.
  Using $|e^{i u} - e^{i v}| \leq |u - v|$ for $u,v \in \mathbb{R}$ yields 
  \begin{equation*}\begin{aligned}
    |(U^q)^* g(x) - (U^q)^* g(y)| \stackrel{\eqref{eq:auxEq25}}{=}&  \frac{1}{2\pi}\left| \int_\mathbb{R} \frac{\hat{g}(\xi)(e^{-i x \xi}-e^{-i y \xi})}{q - \eta(\xi)} \dd \xi \right| \leq  \frac{1}{2 \pi} |x - y| \int_\mathbb{R} \frac{|\xi| |\hat{g}(\xi)|}{|q - \eta(\xi)|} \dd \xi \\ \stackrel{\eqref{eq:symbolEstimate1}}{\leq}& \frac{1}{2 \pi} |x - y| \int_\mathbb{R} \frac{|\xi| |\hat{g}(\xi)|}{| \eta(\xi)|} \dd \xi \stackrel{\eqref{eq:auxEq24}}{=} \frac{L}{2 \pi} |x - y|
  \end{aligned}\end{equation*}
  for any $q > 0$ and $x, y \in \mathbb{R}$. Letting $q \to 0$ and using that $(U^q)^* g \to V^* g$ pointwise (even in $C_0(\mathbb{R})$) by \eqref{eq:potentialDomainChar}, this last estimate implies the result.
\end{proof}

The next result shows that if a solution of the Poisson equation exists (e.g. if the conditions of Proposition~\ref{p1} hold, but here we impose a slightly weaker assumption), then $H \geq 0$ implies $H \in L^1(\R)$. This is useful for verifying the conditions of Theorem~\ref{thm:main}.

\begin{proposition}
  If $g \in C_0(\mathbb{R}) \cap L^1(\mathbb{R})$, $\xi\mapsto\frac{\hat{g}(\xi)}{\eta(\xi)}$ is locally bounded at zero,  there is a solution $H \in C_0(\R)$ to the Poisson equation $\mathcal A^*H=g$ and $H \geq 0$, then $H\in L^1(\R)$.
\end{proposition}
 
\begin{proof} 
  Using \eqref{eq:ResolventFourier} which did not rely on the stronger assumptions of Proposition \ref{p1} and the local boundedness of $\xi\mapsto\frac{\hat{g}(\xi)}{\eta(\xi)}$ in the first and \eqref{eq:symbolEstimate1} in the second inequality, there is some $\varepsilon>0$ and $C>0$ with
  \begin{equation}\label{eq:auxEq18}
    |\widehat{(U^q)^* g}(\xi)| \leq C |(q - \eta(\xi))|^{-1} |\eta(\xi)| \leq C\quad \text{for } \xi \in B_\varepsilon(0).
  \end{equation}
  Let us take a function $\varphi$ such that
  \begin{equation}\label{eq:testFct} 
    \varphi \in L^1(\mathbb{R}) \text{ satisfies } \varphi = 0 \text{ on } \mathbb{R} \setminus B_\varepsilon(0), \varphi \geq 0,  \hat{\varphi} \in L^1(\mathbb{R}) \text{ and }\hat{\varphi} \geq 0.
  \end{equation}
  Since $\varphi$, $\hat{\varphi}, (U^q)^* g \in L^1(\mathbb{R})$, see the beginning of the proof of Proposition \ref{p1}, Fubini's Theorem gives 
  \begin{equation}\label{eq:ResolventIntegralDamped}
    \int_\mathbb{R} (U^q)^* g(x) \hat{\varphi}(x) \dd x = \int_\mathbb{R} \int_\mathbb{R} (U^q)^* g(x) e^{i \xi x} \varphi(\xi) \dd \xi \dd x = \int_\mathbb{R} \widehat{(U^q)^* g}(\xi)\varphi(\xi) \dd \xi . 
  \end{equation}
  Furthermore, $H \hat{\varphi} \geq 0$ and $H \hat{\varphi} \in L^1(\mathbb{R})$ since $H \in C_0(\mathbb{R})$ and we have assumed $H \geq 0$ and \eqref{eq:testFct}. Thus, recalling $H = -V^* g$, we may estimate
  \begin{equation}\label{eq:auxEq19}\begin{aligned}
    0 \leq \int_\mathbb{R} H(x)\hat{\varphi}(x) \dd x 
    &\stackrel{\eqref{eq:potentialDomainChar}}{=}  - \lim_{q \to 0}  \int_\mathbb{R} (U^q)^* g(x) \hat{\varphi}(x) \dd x  \\
    &\stackrel{\eqref{eq:ResolventIntegralDamped}}{=} -\lim_{q \to 0}  \int_\mathbb{R} \widehat{(U^q)^* g}(\xi)\varphi(\xi) \dd \xi 
    \stackrel{\eqref{eq:auxEq18}}{\leq} C \int_\mathbb{R} \varphi(\xi) \dd x, 
  \end{aligned}\end{equation}
  where the first equality uses dominated convergence and the last step relies on our assumption~\eqref{eq:testFct} that $\varphi = 0$ outside $B_\varepsilon(0)$. 

  We now claim that there exists $\{\varphi_n\}_{n \in \mathbb{N}}\subset L^1(\mathbb{R})$ and $I \in (0,\infty)$ such that for each $n$, $\varphi_n$ satisfies \eqref{eq:testFct} and $\lim_{n \to \infty} \hat{\varphi_n}(x) = I $ for any $x \in \mathbb{R}$. Assuming that such a sequence can indeed be constructed, the following argument will finish the proof: inserting $\varphi_n$ in \eqref{eq:auxEq19}, letting $n \to \infty$ and using Fatou's Lemma yields
  \begin{equation*}\begin{aligned}
    0 \leq \int_\mathbb{R} H(x) \dd x & = \frac{1}{I} \int_\mathbb{R} \liminf_{n \to \infty} H(x) \hat{\varphi_n}(x) \dd x \\ &\leq  \liminf_{n \to \infty}  \frac{1}{I} \int_\mathbb{R} H(x) \hat{\varphi_n}(x) \dd x \\ &\leq \liminf_{n \to \infty} \frac{C}{I} \int_\mathbb{R} \varphi_n(x) \dd x = \liminf_{n \to \infty} \frac{C}{I} \hat{\varphi_n}(0)  = C < \infty
  \end{aligned}  \end{equation*}
  and therefore indeed $H \in L^1(\mathbb{R})$.
  \smallskip

  Thus, the remainder of the proof will be devoted to construct a sequence $\{\varphi_n\}_{n \in \mathbb{N}}\subset L^1(\mathbb{R})$ with the desired properties.
  Take $\chi_0 \in C_c^\infty(\mathbb{R}) \setminus \{0\}$ with $\chi_0 \geq 0$, $\chi_0(-x) = \chi_0(x)$ for all $x \in \mathbb{R}$ and $\chi_0(x) = 0$ for $x \notin B_{\varepsilon/3}(0)$. Set $\chi(y) := \int_\mathbb{R} \chi_0(y-x)\chi_0(x) \dd x = \chi_0 * \chi_0 (y) $. Then $\chi(y) = 0$ for $y \notin B_\varepsilon(0)$ and since the Fourier transform turns convolution into products, $\hat{\chi}(\xi) = (\hat{\chi_0}(\xi))^2$ for all $\xi \in \mathbb{R}$. In particular, $\hat{\chi} \geq 0$. Furthermore, $\hat{\chi_0} \neq 0$ implies that $I:=\int_\mathbb{R} \hat{\chi}(\xi) \dd \xi = \int_\mathbb{R} (\hat{\chi_0}(\xi))^2 \dd \xi$ satisfies $I > 0$. Since the Fourier transform maps the space $\mathcal{S}$ of rapidly decreasing functions into itself, $\chi_0 \in C_c^\infty(\mathbb{R}) \subset \mathcal{S}$ implies $\hat{\chi_0} \in \mathcal{S} \subset L^2(\mathbb{R})$. Thus $\hat{\chi}= (\hat{\chi_0})^2$ is integrable and we obtain $I \in (0,\infty)$. Finally, since $\hat{\chi} \in L^1(\mathbb{R})$, Fourier inversion gives $\chi(x) = (2\pi)^{-1} \hat{\hat{\chi}}(-x)$ for all $x \in \mathbb{R}$ (see \cite[Prop.~37.2]{Sato1999}). 
  
  For $n \in \mathbb{N}$, define 
  \begin{equation*}
    \varphi_n (x) := 2 \pi n \chi(-x) \exp\bigg(-\frac{1}{2}n^2 x^2 \bigg),  \quad x \in \mathbb{R},
  \end{equation*}
  and note that $\varphi_n(x) = \hat{\hat{\chi}}(x) n \hat{\psi_n}(x) $, where $\psi_n(x):= \frac{1}{n\sqrt{2 \pi}}\exp(-x^2/(2n^2))$ is the density of a normal with mean zero and variance $n$. In particular, $\varphi_n =n \widehat{(\hat{\chi} * \psi_n)}$ and using $\hat{\hat{f}}(x) = f(-x) $ for $f \in L^1(\mathbb{R})$ with $\hat{f} \in L^1(\mathbb{R})$, as above we obtain
  \begin{equation}\label{eq:testFctFT} 
    \hat{\varphi_n}(\xi) = n \widehat{\widehat{(\hat{\chi} * \psi_n)}}(\xi) = n \hat{\chi}* \psi_n(-\xi) = 
    \frac{1}{\sqrt{2 \pi}} \int_\mathbb{R} \hat{\chi}(y)\exp\bigg(-\frac{1}{2 n^2}(\xi + y)^2\bigg) \dd y .
  \end{equation}
  Thus, for any $n \in \mathbb{N}$, $\varphi_n$ indeed satisfies \eqref{eq:testFct} and applying dominated convergence (and noting that the integrand on the right-hand side converges pointwise to $\hat{\chi}$) in \eqref{eq:testFctFT} gives for any $\xi \in \mathbb{R}$, $ \lim_{n \to \infty} \hat{\varphi_n}(\xi) =  I$ as desired.
\end{proof}

\subsection{Necessity of Conditions}\label{sec:necessity}

In this section we assume that $\tau$ is a finite mean solution for the Skorokhod embedding problem corresponding to $\mu_i(\d x)=h_i(x)d x$ and study the associated Poisson equation $\mathcal A^* H=h_1-h_0$. We show that
\begin{itemize}
\item	there is a solution $H$. Using Proposition \ref{p1} we need to show $(\hat h_1-\hat h_0)/\eta\in L^1(\R)$. Integrability at infinity is only a consequence of the smoothness assumption on $h_0, h_1$ (Lemma \ref{lemma}) without using $\tau$. Integrability at zero is a consequence of Dynkin's formula and the existence of $\tau$ (Lemma \ref{lem}). 
\item $H\geq 0$ and $H\in L^1(\R)$. This is a consequence of Dynkin's formula and the Riesz representation theorem.
\end{itemize}

The first lemma is the source of our assumptions on the regularity for $h_0$ and $h_1$.
\begin{lemma}\label{lemma}
  Under Assumption~\ref{ass:regularity} on $h_0$ and $h_1$ (as in Theorem~\ref{thm:main}) we have
  \begin{align*}
    \int_{|u|>R} \left| \frac{\widehat{h_1}(u) - \widehat{h_0}(u)}{\eta(u)} \right| \dd u <\infty
  \end{align*}
  for some $R>0$.
\end{lemma}

\begin{proof}
  We consider all cases listed in Assumption~\ref{ass:regularity} separately. \smallskip

  \textbf{Type S:} 
  Since $|\widehat{h_i}(u)|\leq 1$ for $i=0,1$, for $R \geq 1$ one can use symmetry and the integrability assumption for $1/\eta$ to estimate
  \begin{align*}
    \int_{|u|>R} \left| \frac{\widehat{h_1}(u) - \widehat{h_0}(u)}{\eta(u)} \right| \dd u  \leq \int_{|u|>R} \frac{2}{|\eta(u)|} \dd u  \leq 4 \int_R^\infty \frac{1}{|\eta(u)|} \dd u <\infty.
  \end{align*}
  \smallskip
	
  \textbf{Type 0:}
  By assumption there exists $R > 0, C > 0$ with
  \begin{equation}\label{eq:symbolatInfinity} 
    |\eta(u)| \geq C \quad \text{ for } \quad |u| \geq R.  
  \end{equation}
  On the other hand, the regularity assumptions for type 0 guarantee that $h_i^{(2)} \in L^1(\R)$ and so standard Fourier analysis gives
  \begin{equation}\label{eq:FourieratInfinity}
    \left|u^{2} (\widehat{h_1}(u) - \widehat{h_0}(u))\right| = \left| \widehat{h_1^{(2)}}(u) - \widehat{h_0^{(2)}}(u)\right| \leq \tilde{C}
  \end{equation}
  for all $u \in \R$, where $\tilde{C}:=\int_\R |h_1^{(2)}(x)|+|h_0^{(2)}(x)| \dd x$.

  Using \eqref{eq:symbolatInfinity} in the first and \eqref{eq:FourieratInfinity} in the second step yields
  \begin{align*}
    \int_{|u|>R} \left| \frac{\widehat{h_1}(u) - \widehat{h_0}(u)}{\eta(u)} \right| \dd u  \leq \frac{1}{C} \int_{|u|>R} \left|\widehat{h_1}(u) - \widehat{h_0}(u)\right| \dd u  \leq \frac{\tilde{C}}{C} \int_{|u|>R} \frac{1}{|u|^2} \dd u <\infty.
  \end{align*}
  \smallskip	
	
  \textbf{Type D:} 
  By assumption there exists $R>0$ such that $\widehat{h_1}(u)-\widehat{h_0}(u) = 0$ for all $|u| > R$ and so the integral is $0$.
\end{proof}

Lemma \ref{lemma} was independent of the Skorokhod embedding problem whereas the integrability around the origin indeed is a consequence of the SEP. The crucial ingredient of the proof is the use of Dynkin's formula for the complex exponential function:

\begin{lemma}\label{lem}
  Suppose $\tau$ is a finite mean solution to the Skorokhod embedding problem for $\mu_i(\d x)=h_i(x)\dd x$. Then
  \begin{align}\label{eq:auxEq50}
    \eta(u) \E^{\mu_0} \left[ \int_0^{\tau} e^{i u L_s} \dd s \right]= \widehat{h_1}(u) - \widehat{h_0}(u)
  \end{align}
  for all $u \in \R$ and 
    \begin{align*}
    \int_{|u|\leq R} \left| \frac{\widehat{h_1}(u) - \widehat{h_0}(u)}{\eta(u)} \right| \dd u <\infty
  \end{align*}
  for any $R>0$.

\end{lemma}

\begin{proof}
  Let $f, g \in C_b(\R)$ be such that
  \begin{equation*}
    M^f_t := f(L_t) - f(L_0) - \int_0^t g (L_s) \dd s,  \quad t \geq 0,
  \end{equation*}
  is a martingale. The optional sampling theorem implies that also $(M^f_{t \wedge \tau})_{t \geq 0}$ is a martingale. In particular, for any $t \geq 0$,
  \[ 
    \E^{\mu_0} \left[ \int_0^{\tau \wedge t} g(L_s) \dd s \right] = \E^{\mu_0}[f(L_{\tau \wedge t})] - \E^{\mu_0}[f(L_0)].  
  \]
  Since $\P^{\mu_0}(\tau < \infty)=1$ and $L$ is quasi-left continuous, it holds that $\P^{\mu_0}(\lim_{t \to \infty} L_{\tau \wedge t} =L_\tau)=1$.
  Using dominated convergence, $\E^{\mu_0}[\tau] < \infty$ and $f, g  \in C_b(\mathbb{R})$, one may let $t \to \infty$ to obtain Dynkin's formula,
  \begin{equation}\label{eq:dynkin} 
    \E^{\mu_0} \left[ \int_0^{\tau} g(L_s) \dd s \right] = \E^{\mu_0}[f(L_{\tau})] - \E^{\mu_0}[f(L_0)] .
  \end{equation}
  
  For $u \in \R$, set 
  \begin{equation*}
    M^u_t:=e^{i u L_t} - e^{i u L_0} - \eta(u)\int_0^t e^{i u L_r} \dd r ,  \quad t \geq 0.
  \end{equation*}
  Then $M^u$ is $(\mathcal{F}_t)_{t \geq 0}$-adapted and for any $t \geq 0$, $M_t^u$ is a bounded random variable. Furthermore, 
  \begin{align*}
    \E^{\mu_0}[M^u_t-M^u_s|\mathcal{F}_s]
    &=e^{i u L_s}\E^{\mu_0}\left[\left.e^{i u (L_t-L_s)} - 1 -  \eta(u) \int_s^t e^{i u (L_r-L_s)} \dd r \right|\mathcal{F}_s \right] \\
    &\stackrel{\mathmakebox[\widthof{=}]{\eqref{eq:charFctLevy}}}{=}e^{i u L_s}\bigg(e^{(t-s)\eta(u)} - 1 - \eta(u) \int_s^t e^{ (r-s) \eta(u)} \dd r\bigg) \\ &= 0
  \end{align*}
  and therefore $M^u$ is a complex-valued $(\mathcal{F}_t)_{t \geq 0}$-martingale. Thus Dynkin's formula \eqref{eq:dynkin} can be applied to $f(x):= e^{i u x}$, $g(x):= \eta(u) f(x)$ and so
  \[ \eta(u) \E^{\mu_0} \left[ \int_0^{\tau} e^{i u L_r} \dd r \right] = \E^{\mu_0} \left[ \int_0^{\tau} g(L_r) \dd r \right] = \E^{\mu_0}[f(L_{\tau})] - \E^{\mu_0}[f(L_0)] = \E^{\mu_0}[e^{i u L_\tau}] - \E^{\mu_0}[e^{i u L_0}].  \]
This proves the first claim of the lemma. We can now deduce that $(\hat h_1-\hat h_0)/\eta$ is integrable in compact sets. By the above we obtain
  \begin{align*}
    \left|\widehat{h_1}(u) - \widehat{h_0}(u) \right| 
    = \left|\eta(u) \E^{\mu_0} \left[ \int_0^{\tau} e^{i u L_r} \dd r \right]\right| 
    \leq |\eta(u)|\E^{\mu_0}[\tau],
  \end{align*}
and this implies
  \begin{align*}
    \int_{|u|\leq R} \left| \frac{\widehat{h_1}(u) - \widehat{h_0}(u)}{\eta(u)} \right| \dd u \leq 2 R\, \E^{\mu_0}[\tau]<\infty.
  \end{align*} 
\end{proof}

Combining the previous lemmas we proved that a finite mean solution to the Skorokhod embedding problem for "sufficiently smooth" densities implies $(\widehat{h_1}-\widehat{h_0})/\eta\in L^1(\R)$ which, solving in Fourier domain, implies there is a solution $H$ to the Poisson equation $\mathcal A^* H=h_1-h_0$.\smallskip

Now we can finish the proof by showing that existence of a finite mean solution to the Skorokhod embedding problem implies $H\geq 0$ and $H\in L^1(\R)$.

\begin{proof}[Proof of Theorem~\ref{thm:main} (necessity):]
  We showed  that  $(\widehat{h_1}-\widehat{h_0})/\eta\in L^1(\R)$ so the Poisson equation $\mathcal A^*H=h_1-h_0$ can be solved using Proposition \ref{p1} as
  \begin{align*}
    H(x)= \frac{1}{2 \pi} \int_\mathbb{R} \frac{\widehat{h_1}(\xi)-\widehat{h_0}(\xi)}{\eta(\xi)}e^{-i x \xi} \dd \xi,\quad x \in \mathbb{R}.
  \end{align*}
  It remains to prove $H\geq 0$ and $H\in L^1(\R)$:\smallskip

  Define the functional $\Lambda \colon C_c(\R) \to \R$ by
  \begin{equation*} 
    \Lambda(g):= \E^{\mu_0} \left[ \int_0^{\tau} g(L_s) \dd s \right].
  \end{equation*}
  Then $\Lambda(g)\geq 0$ for $g \geq 0$, $\Lambda$ is linear and $|\Lambda(g)|\leq \|g\|_\infty \E^{\mu_0}[\tau]$. By the Riesz Representation Theorem (e.g. \cite[Theorem~2.14]{Rudin1987}), there exists a measure $\nu$ on $\mathcal{B}(\R)$ such that for all $g \in C_c(\R)$, 
  \begin{equation}\label{eq:auxEq55} 
    \E^{\mu_0} \left[ \int_0^{\tau} g(L_s) \dd s \right] = \Lambda(g) = \int_\R g(x)\, \nu(\d x). 
  \end{equation}

  Choosing $\{g_n\}_{n \in \mathbb{N}} \subset C_c(\mathbb{R})$, increasing monotonically to $1$ with $g_n \geq 0$ and applying monotone convergence gives
  \begin{equation*}
    \nu(\mathbb{R}) = \lim_{n \to \infty} \int_\mathbb{R} g_n(x)\, \nu(\d x) = \lim_{n \to \infty} \E^{\mu_0} \left[ \int_0^{\tau} g_n (L_s) \dd s \right] = \E^{\mu_0}[\tau] < \infty.
  \end{equation*}
  Thus $\nu$ is a finite measure and by dominated convergence, \eqref{eq:auxEq55} holds for all $g \in C_b(\R)$. Inserting $g(x) := e^{i u x}$ for $u \in \R$ in \eqref{eq:auxEq55} and using \eqref{eq:auxEq50} yields
  \begin{equation*}
    \hat{\nu}(u) = \int_\R e^{i u x}\, \nu(\d x) = \E^{\mu_0} \left[ \int_0^{\tau} e^{i u L_s} \dd s \right]=\frac{\widehat{h_1}(u) - \widehat{h_0}(u)}{\eta(u)},
  \end{equation*}
  which is integrable by Lemma~\ref{lemma} and Lemma~\ref{lem}.
  Hence, e.g. by \cite[Proposition~2.5~(xii)]{Sato1999}, $\nu$ is absolutely continuous with respect to the Lebesgue measure and has a (non-negative) bounded continuous density given by
  \begin{equation*} 
    x \mapsto \frac{1}{2 \pi} \int_\mathbb{R} \frac{\widehat{h_1}(\xi)-\widehat{h_0}(\xi)}{\eta(\xi)}e^{-i x \xi} \dd \xi,\quad x \in \mathbb{R}. 
  \end{equation*}
  But this function is identical to $H$ and thus $\nu(\d x) = H(x)\dd x$. In particular, $H \geq 0$ and $H \in L^1(\R)$.
\end{proof}

\subsection{Sufficiency of Conditions}\label{sec:sufficiency}

Under the assumptions of Theorem~\ref{thm:main} we now construct a finite mean stopping time with $L_\tau\sim h_1(x)\dd x$ under the initial condition $L_0\sim h_0(x)\dd x$. We refer the reader to the sketch in Subsection~\ref{sec:summary} to follow more easily the construction of $\tau$. During the proof we refer to the time-change arguments and uniqueness results for Fokker-Planck equations from Section~\ref{sec:Fokker-Planck}.\smallskip

Let $\mathcal{D}:=C_c^\infty(\mathbb{R})$ and define the action of the L\'evy generator $\mathcal{A}\colon \mathcal{D} \to C_0(\mathbb{R})$ for $u \in \mathcal{D}$ via~\eqref{eq:LevyGenerator}. Furthermore, taking into account the definitions from Theorem~\ref{thm:main}, let
\begin{align}\label{eq:phiDef} 
  \phi(t,x) &:=  (1-t)h_0(x) + t h_1(x), \\ 
  \label{eq:sigmaDef} \sigma(t,x) &:=  \frac{H(x)}{(1-t)h_0(x) + t h_1(x)} ,\quad (t,x) \in [0,1]\times \mathbb{R},
\end{align}
and, under $\P^{\mu_0}$,
\begin{equation}\label{eq:Moserdelta}
  \delta(s) := \inf\big\{t \in [0,\rho) \, :\, \Delta(t)\geq s\big \} \wedge \rho , \quad s \in [0,1],
\end{equation}
where
\begin{equation}\label{eq:MoserDelta}
  \Delta(t) := 1-e^{G(t)}+\int_0^t e^{(G(t)-G(r))} \frac{h_1(L_r)}{H(L_r)} \dd r, \quad  t \in [0,\rho),
\end{equation}
with 
\begin{align} \label{eq:RhoAndGProof}
  \rho :=  \inf\{ t \in [0,\infty) \, : \, H(L_t) = 0 \}\quad\text{ and }\quad G(t) := \int_0^t \frac{h_1(L_r)-h_0(L_r)}{H(L_r)} \dd r.
\end{align}

The proof is split in two main steps: First we assume in addition that $h_0$ and $h_1$ are such that $\sigma$ is bounded and argue as in Section~\ref{sec:summary}. Then, for $\sigma$ unbounded, we approximate $h_i$ by $h_i^{(\varepsilon)}$ with associated $\sigma^{(\varepsilon)}$ bounded and deduce Theorem~\ref{thm:main}.
  
\begin{proof}[Proof of Theorem~\ref{thm:main} (sufficiency \textbf{if \boldmath$\sigma$ is bounded})]
  For the proof the following statements are established: 
  \begin{itemize}
    \item[(i)] $(\delta(s))_{s \in [0,1]}$ constitutes a family of $(\mathcal{F}_t)_{t \geq 0}$-stopping times satisfying $\P^{\mu_0}\text{-}a.s.$, 
               \begin{equation}\label{eq:auxEqEnd}
 	         \delta(s) = \int_0^s \sigma(u,L_{\delta(u)}) \dd u,\quad s \in [0,1], 
               \end{equation}
    \item[(ii)] for any $s \in [0,1]$, the law of $ L_{\delta(s)}$ under $\P^{\mu_0}$ is $\phi(s,x) \dd x$, 
    \item[(iii)] $\E^{\mu_0}[\delta(1)] = \int_\mathbb{R} H(x) \dd x<\infty$.
  \end{itemize}
  
  Theorem~\ref{thm:main} can then be deduced from (i)-(iii) by setting $\tau := \delta(1)$ because $\phi(1,\cdot)=h_1$ by construction. Note that the stopping time looks slightly different here than in the statement of Theorem~\ref{thm:main}. Both representations are equal because from~\eqref{eq:MoserDelta} one obtains 
  \begin{align*}
    \Delta(t) \geq 1 \quad 
    &\Longleftrightarrow \quad -e^{G(t)}+\int_0^t e^{(G(t)-G(r))} \frac{h_1(L_r)}{H(L_r)} \dd r \geq 0 \\ 
    &\Longleftrightarrow\quad \int_0^t e^{-G(r)} \frac{h_1(L_r)}{H(L_r)} \dd r \geq 1,
  \end{align*}
  and, hence, the claimed representation of $\tau=\delta(1)$ as generalized inverse in Theorem~\ref{thm:main} and in~\eqref{eq:Moserdelta} are equal.\smallskip

  The proof of (i)-(iii) proceeds roughly as follows: Having verified in Lemma~\ref{lem:levyFeller} that $L$, $\mathcal{D}$ and $\mathcal A$ fall within the framework of \cite{Nr2}, one may rely on the results from Section~\ref{sec:Fokker-Planck}. Then, firstly, it is proved that $L$ and $\sigma$ satisfy the assumptions of Proposition~\ref{prop:timechange} and that \eqref{eq:MoserDelta} is the solution to the differential equation~\eqref{eq:Delta}, so Proposition~\ref{prop:timechange} implies (i). Based on Proposition~\ref{prop:density time changed}, one then verifies that $(\phi(s,x)\dd x)_{s \in [0,1]}$ and the marginals of $L_{\delta(\cdot)}$ are solutions to the Fokker-Planck equation~\eqref{eq:KolmogorovEqn}. Then from the uniqueness result Theorem~\ref{thm:KolmogorovUniqueness} it follows that $L_{\delta(s)}$ indeed has the law $ \phi(s,x) \dd x$, i.e. (ii). Combining the representation of $\delta(1)$ established in (i) with (ii) and the fact that $\sigma(t,x) \phi(t,x) = H(x)$ for all $t \in [0,1]$ and $x \in \mathbb{R}$ one easily obtains (iii). \smallskip
 
  \textbf{Verification of (i):}  
 The claim is a consequence of Proposition~\ref{prop:timechange} (compare \eqref{eq:Delta} and \eqref{eq:stoppingTime2} for the formula of $\delta$ in terms of $\Delta$). We only need to verify the conditions of Proposition~\ref{prop:timechange}  and then solve \eqref{eq:Delta} 
for our choice of $\sigma$ from \eqref{eq:sigmaDef}. 
It is the particular form of the denominator which allows us to solve \eqref{eq:Delta} explicitly and get the formula for $\Delta$ as in \eqref{eq:MoserDelta}.\smallskip
  
  By Proposition~\ref{p1}, $H \in D(\mathcal{A}^*) \subset C_0(\mathbb{R})$, where $\mathcal{A}^*\colon D(\mathcal{A}^*) \to C_0(\mathbb{R})$ denotes the adjoint (see Subsection~\ref{subsec:levyprelim} and \eqref{eq:adjointDef}). Since also $H \in D(\mathcal{A})$ by Proposition~\ref{prop:levyregularaux}, $H$ is regular for (the law of) $L$ by Proposition~\ref{prop:regular}. Since $h_0$ and $h_1$ are assumed positive and continuous, for any $K \subset \mathbb{R}$ compact, there exist $C_0, C_1 > 0$ such that 
  \begin{equation}\label{eq:auxEq35}
    C_0 \leq h_i(x) \leq C_1, \quad x \in K, \,  i =0,1. 
  \end{equation}  
  Set $\tilde{\sigma}(t,x):=1/\phi(t,x)$ for $(t,x) \in [0,1]\times K$. Then \eqref{eq:auxEq35} implies $1/C_1 \leq \tilde{\sigma}(t,x) \leq 1/C_0$ and 
  \begin{equation*}
    |\tilde{\sigma}(t,x)-\tilde{\sigma}(s,x)|=\frac{1}{\phi(s,x)\phi(t,x)}|\phi(s,x)-\phi(t,x)| \leq \frac{2 C_1}{C_0^2} |t-s|, \quad (t,x) \in [0,1]\times K.
  \end{equation*}
 This is precisely what Assumption~\ref{ass:sigma}~(ii) asks for our $\sigma$ written as $\sigma=H \tilde \sigma$. Since $\sigma$ is also assumed to be bounded in this first part of the proof, Proposition~\ref{prop:timechange} can be applied. The lemma implies that the random times defined by \eqref{eq:stoppingTime2} are stopping times and that \eqref{eq:ode time-change} holds. As the definitions \eqref{eq:stoppingTime2} and \eqref{eq:Moserdelta} coincide, in order to show (i), it thus suffices to show that $\rho$ and $\Delta$ in \eqref{eq:Moserdelta} coincide $\P^{\mu_0}$-a.s. with \eqref{eq:rho} and \eqref{eq:Delta}.\smallskip
  
  Since $H$ is regular for the law of $L$ (as argued above), $\rho$ in \eqref{eq:Moserdelta} (respectively \eqref{eq:RhoAndGProof}) is equal to \eqref{eq:rho} (see Definition~\ref{def:regular}). Furthermore, as shown in  Proposition~\ref{prop:timechange}, the solution to the Carath\'eodory differential equation \eqref{eq:Delta} is $\P^{\mu_0}$-a.s. unique. Thus it suffices to show that $\P^{\mu_0}$-a.s., $\Delta$ defined by \eqref{eq:MoserDelta} is a solution to \eqref{eq:Delta}, i.e. that $\P^{\mu_0}$-a.s.
  \begin{equation}\label{eq:auxEq39}
    \Delta(t) = \int_0^t \sigma(\Delta(r),L_r)^{-1} \dd r,\quad t \in [0, \delta (1)),
  \end{equation}
  holds. Inserting \eqref{eq:sigmaDef} in \eqref{eq:auxEq39} yields
  \begin{equation}\label{eq:auxEq40} 
    \Delta(t) =  \int_0^t \frac{\Delta(r)(h_1(L_r)-h_0(L_r))}{H(L_r)} \dd r + \int_0^t \frac{h_0(L_r)}{H(L_r)} \dd r , \quad t \in [0, \delta (1)).
  \end{equation}
  On the other hand, $\delta(1) \leq \rho$ and $\P^{\mu_0}$-a.s. the candidate solution $\Delta$ in \eqref{eq:MoserDelta} is absolutely continuous on every closed subinterval of $[0,\rho)$ and 
  \begin{align*} 
    \dot{\Delta}(t) & = - \dot{G}(t) e^{G(t)} + \dot{G}(t) \int_0^t e^{(G(t)-G(r))} \frac{h_1(L_r)}{H(L_r)} \dd r + \frac{h_1(L_t)}{H(L_t)} \\
    & = \dot{G}(t)(\Delta(t)-1) + \frac{h_1(L_t)}{H(L_t)}\\
    & = \Delta(t) \frac{h_1(L_t)-h_0(L_t)}{H(L_t)}  + \frac{h_0(L_t)}{H(L_t)},
  \end{align*} 
  for almost every $t \in [0,\rho)$. This is equivalent to \eqref{eq:MoserDelta} being a solution to \eqref{eq:auxEq40} on $[0,\rho)$ as desired. 
  \smallskip
 
  \textbf{Verification of (ii):}   
  Firstly, in (i) it has been verified that Assumption~\ref{ass:sigma} holds. Secondly, Assumption~\ref{ass:SigmaBounded} holds and, as argued above, Assumption~\ref{ass:Hregular} is satisfied.  Thus, Proposition~\ref{prop:density time changed} and Theorem~\ref{thm:KolmogorovUniqueness} can be applied. This shows that $(\tilde{p}(s,\cdot))_{s \in [0,1]}$ is the unique solution to the Fokker-Planck equation ~\eqref{eq:KolmogorovEqn}, where $\tilde{p}(s,\cdot)$ is the law of $L_{\delta(s)}$ under $\P^{\mu_0}$. Thus, in order to establish (ii), it suffices to verify that $(p(s,\cdot))_{s \in [0,1]}$ with $p(s,\d x) = \phi(s,x)\dd x$ also satisfies the Fokker-Planck equation~\eqref{eq:KolmogorovEqn}.

  Inserting $p(s,\d x) = \phi(s,x)\dd x$ with $\phi$ from \eqref{eq:phiDef} into the left-hand side of \eqref{eq:KolmogorovEqn}, using $H \in D(\mathcal{A}^*)$, $\mathcal{A}^* H = h_1-h_0$ (by \eqref{b} and Proposition~\ref{p1}) and $H \in L^1(\mathbb{R})$, $ h_1-h_0 \in L^1(\mathbb{R})$ (by assumption), Lemma~\ref{lem:adjoint} gives
  \begin{align*}
    \int_\mathbb{R} f(x)\,p(t,\d x) - \int_\mathbb{R} f(x)\, \mu_0(\d x) &\stackrel{\mathmakebox[\widthof{=}]{\eqref{eq:phiDef}}}{=} \int_\mathbb{R} f(x) t (h_1-h_0)(x)\dd x \\
    &= \int_0^t \int_\mathbb{R} \mathcal{A}^*H(x) f(x) \dd x \dd s
    \\ &\stackrel{\mathmakebox[\widthof{=}]{\eqref{eq:IntegralAdjoint}}}{=} \int_0^t \int_\mathbb{R} H(x) \mathcal{A}f(x) \dd x \dd s\\
    & = \int_0^t \int_\mathbb{R} \sigma(s,x)\mathcal{A}f(x)\,p(s,\d x) \dd s, 
   \end{align*}
  where the last step is just the definition \eqref{eq:sigmaDef} and \eqref{eq:phiDef}. Hence, by Theorem~\ref{thm:KolmogorovUniqueness} and our argument above we may indeed conclude (ii). 
  \smallskip
  
  \textbf{Verification of (iii):}
  Having verified that the marginals of $L_{\delta(\cdot)}$ are given as in \eqref{eq:phiDef}, we may apply the representation of $\delta$ as solution to an integral equation (established in (i)), Tonelli's Theorem and the definition of $\sigma$ and $\phi$ to see, using \eqref{eq:auxEqEnd},
  \begin{equation*}\begin{aligned}
    \mathbb{E}^{\mu_0}[\delta(1)] &=\mathbb{E}^{\mu_0}\bigg[\int_0^1 \sigma(u,L_{\delta(u)}) \dd u \bigg ]=  \int_0^1 \E^{\mu_0}[\sigma(u,L_{\delta(u)})] \dd u = \int_0^1 \int_\mathbb{R} \sigma(u,x) \phi(u,x) \dd x \dd u \\ & = \int_\mathbb{R} H(x) \dd x.
  \end{aligned}\end{equation*}
  The right-hand side is finite by assumption and so $\tau=\delta(1)$ has finite mean.
\end{proof}

\begin{proof}[Proof of Theorem~\ref{thm:main} (sufficiency)]  
  To finish the proof of sufficiency we need to remove the assumption that $h_0$ and $h_1$ are such that $\sigma=H/\phi$ is bounded using a truncation procedure. For this sake we shift  the densities in order to shift down $\phi$. \smallskip
	
  \textbf{Approximate stopping times \boldmath$\delta^{(\varepsilon)}$:}  Since $H \in L^1(\mathbb{R})$ by assumption, $C:= \int_\mathbb{R} H(x)\dd x$ is well-defined and $p := H / C$ is a probability density on $\mathbb{R}$. For any $\varepsilon \in (0,1)$ we define
  \begin{align}\label{eq:translatedDens}
    &h_i^{(\varepsilon)}(x):= (1-\varepsilon) h_i(x) + \varepsilon p(x),\quad i=0,1,\nonumber\\
    &\phi^{(\varepsilon)}(t,x) := (1-t)h_0^{(\varepsilon)}(x) + t h_1^{(\varepsilon)}(x),\\
    &H	^{(\varepsilon)}(x):=(1-\varepsilon)H(x),\nonumber
  \end{align}
  for $(t,x) \in [0,1]\times \mathbb{R}$, and the approximation to $\sigma$ by $\sigma^{(\varepsilon)}:=\frac{H^{(\varepsilon)}}{\phi^{(\varepsilon)}}$. Then, for any $\varepsilon \in (0,1)$,  $\phi^{(\varepsilon)}\geq \varepsilon p$ and so, for any $t\in [0,1]$ and $x \in \mathbb{R}$ with $H(x) \neq 0$, 
  \begin{equation*}
    \sigma^{(\varepsilon)}(t,x) \leq \frac{H^{(\varepsilon)}(x)}{\varepsilon p(x)}  = \frac{(1-\varepsilon) C}{\varepsilon}.
  \end{equation*}
  Thus, for any $\varepsilon \in (0,1)$, $\sigma^{(\varepsilon)}$ is bounded. Furthermore, $h_i^{(\varepsilon)} \in C_0(\mathbb{R})$, $h_i^{(\varepsilon)}(x) > 0$ and 
  \begin{equation}\label{eq:auxEq41} 
    h_1^{(\varepsilon)}(x) - h_0^{(\varepsilon)}(x) = (1-\varepsilon) (h_1(x) - h_0(x)) 
  \end{equation}
  for any $x \in \mathbb{R}$, $i=0,1$ and $\varepsilon \in [0,1)$. In particular, $H^{(\varepsilon)}$ satisfies the Poisson equation $\mathcal A^* H^{(\varepsilon)}=h_1^{(\varepsilon)}-h_0^{(\varepsilon)}$.
  Since $H^{(\varepsilon)}=(1-\varepsilon)H$, the following properties are inherited from $H$: $H^{(\varepsilon)}$ is non-negative, $H^{(\varepsilon)} \in L^1(\mathbb{R})$ and $H^{(\varepsilon)} \in D(\mathcal{A})$ by Proposition~\ref{prop:levyregularaux}.
  
  Thus Step (i) of the bounded case applied with $h_0^{(\varepsilon)}, h_1^{(\varepsilon)}$ instead of $h_0, h_1$ shows that, for any $\varepsilon \in (0,1)$,
  \begin{equation}\label{eq:auxEq43} \begin{aligned}
    \delta^{(\varepsilon)}(s)  := \inf\{t \in [0,\rho) \, :\, \Delta^{(\varepsilon)}(t)\geq s \} \wedge \rho , \quad s \in [0,1],
  \end{aligned} \end{equation}
  constitutes a family of $(\mathcal{F}_t)_{t \geq 0}$-stopping times, where 
  \begin{align}\label{eq:auxEq42}
    &\rho^{(\varepsilon)}:=  \inf\{ t \in [0,\infty) \, : \, H^{(\varepsilon)}(L_t) = 0 \},\nonumber\\
    &\Delta^{(\varepsilon)}(t)  := 1-e^{G^{(\varepsilon)}(t)}+\int_0^t e^{(G^{(\varepsilon)}(t)-    G^{(\varepsilon)}(r))} \frac{h_1^{(\varepsilon)}(L_r)}{H^{(\varepsilon)}(L_r)} \dd r,\\
    &G^{(\varepsilon)}(t)  := \int_0^t \frac{h_1^{(\varepsilon)}(L_r)-h_0^{(\varepsilon)}(L_r)}{H^{(\varepsilon)}(L_r)} \dd r,\nonumber
  \end{align}
  for $t \in [0,\rho)$ and we note that $\rho=\rho^{(\varepsilon)}$, since $H^{(\varepsilon)}=(1-\varepsilon)H$.
  
  \smallskip

  \textbf{Some simplifications:} 
  The choice of $\delta^{(\varepsilon)}$ is very convenient as the mean and $\Delta^{(\varepsilon)}$ simplify in a neat way. By Step (ii) of the bounded case, for any $\varepsilon \in (0,1)$, $L_{\delta^{(\varepsilon)}(s)}$ has law $\phi^{(\varepsilon)}(s,x)\dd x$ under $\P^{\mu_0}$, for any $s \in [0,1]$ and by (iii),
  \begin{equation}\label{eq:auxEq53}
    \E^{\mu_0}[\delta^{(\varepsilon)}(1)] = \int_\mathbb{R}H^{(\varepsilon)}(x) \dd x = (1-\varepsilon) \int_\mathbb{R} H(x) \dd x.
  \end{equation}
 
  Next, \eqref{eq:auxEq41} and $H^{(\varepsilon)} = (1-\varepsilon) H$ imply that $G^{(\varepsilon)}=G$ and thus from \eqref{eq:auxEq42} one obtains, for any $t \in [0,\rho)$, a simple formula for $\Delta^{(\varepsilon)}$:
  \begin{equation}\begin{aligned} 
    \Delta^{(\varepsilon)}(t) & = 1-e^{G(t)}+\int_0^t e^{(G(t)-G(r))} \frac{h_1^{(\varepsilon)}(L_r)}{(1-\varepsilon)H(L_r)} \dd r \\ &\stackrel{\mathmakebox[\widthof{=}]{\eqref{eq:translatedDens}}}{=} 1-e^{G(t)}+\int_0^t e^{(G(t)-G(r))} \frac{h_1(L_r)}{H(L_r)} \dd r + \frac{\varepsilon}{(1-\varepsilon)C} e^{G(t)}\int_0^t e^{-G(r)}\dd r 
    \\ \label{eq:auxEq45} & = \Delta(t) + \frac{\varepsilon}{(1-\varepsilon)C} e^{G(t)}\int_0^t e^{-G(r)}\dd r.
  \end{aligned} \end{equation}

  \textbf{Limiting stopping time \boldmath$\delta$:} 
  Set $\Delta^{(0)}:=\Delta$ and $\delta^{(0)}:=\delta$ (from \eqref{eq:Moserdelta}). We need to show that $\delta^{(0)}$ is a stopping time and we need to compute the distribution of $L_{\delta}$.\smallskip

  Since $f\colon [0,1) \to \mathbb{R}$, $f(x):=x/(1-x)$, is increasing and $f(0)=0$ the last decomposition of $\Delta^{(\varepsilon)}$ shows that $\P^{\mu_0}$-a.s. for any $0 \leq \varepsilon < \tilde{\varepsilon} < 1$ and all $t \in [0,\rho)$, $\Delta^{(\varepsilon)}(t) \leq \Delta^{(\tilde{\varepsilon})}(t)$ and hence, from \eqref{eq:auxEq43}, $\delta^{(\varepsilon)}(s) \geq  \delta^{(\tilde{\varepsilon})}(s)$ for all $s \in [0,1]$.
  In particular, for any $s \in [0,1]$, $(\delta^{(\frac{1}{n})}(s))_{n \in \mathbb{N}}$ is a sequence of stopping times with $\delta^{(\frac{1}{n})}(s) \leq \delta^{(\frac{1}{n+1})}(s) \leq \delta(s)$ for any $n \in \mathbb{N}$. Thus $\tilde{\delta}(s):=\lim_{n \to \infty} \delta^{(\frac{1}{n})}(s) \in [0,\infty]$ exists $\P^{\mu_0}$-a.s. and $\tilde{\delta}(s) \leq \delta(s)$. Since $(\mathcal{F}_t)_{t \geq 0}$ is right-continuous, $\tilde{\delta}(s)$ is an $(\mathcal{F}_t)_{t \geq 0}$-stopping time. By \eqref{eq:auxEq53} and monotone convergence,
  \begin{equation}\label{eq:auxEq54}
    \E^{\mu_0}[\tilde{\delta}(1)] = \lim_{n \to \infty} \E^{\mu_0}[\delta^{(\frac{1}{n})}(1)] = \int_\mathbb{R} H(x) \dd x < \infty.
  \end{equation}
  In particular, $\tilde{\delta}(s)\leq \tilde{\delta}(1) < \infty$, $\P^{\mu_0}$-a.s.
  
  If $\tilde{\delta}(s) \geq \rho$, then $\delta(s) \leq \rho$ and $\tilde{\delta}(s) \leq \delta(s)$ imply $\tilde{\delta}(s) = \delta(s)$. Otherwise $\tilde{\delta}(s) < \rho$ and thus
  \begin{equation}\label{eq:auxEq46}
    \int_0^{\tilde{\delta}(s)} \frac{1}{H(L_s)} \dd s < \infty.
  \end{equation}  
  Using continuity and the decomposition~\eqref{eq:auxEq45} one obtains
  \begin{equation}\begin{aligned}\label{eq:auxEq48} 
    \Delta(\tilde{\delta}(s)) & = \lim_{n \to \infty} \Delta(\delta^{(\frac{1}{n})}(s)) = \lim_{n \to \infty}\left( \Delta^{(\frac{1}{n})}(\delta^{(\frac{1}{n})}(s)) -  \frac{e^{G(\delta^{(\frac{1}{n})}(s))}\frac{1}{n}}{(1-\frac{1}{n})C} \int_0^{\delta^{(\frac{1}{n})}(s)} e^{-G(r)}\dd r \right) \\ & \geq s, \end{aligned}
  \end{equation}
  where the last inequality follows from  $\Delta^{(\frac{1}{n})}(\delta^{(\frac{1}{n})}(s)) \geq s$ (by definition), the fact that on $\{\tilde{\delta}(s) < \rho\}$, $G$ is bounded on the compact interval $[0,\tilde{\delta}(s)]$ (which follows directly from \eqref{eq:auxEq46}) and $\delta^{(\frac{1}{n})}(s) \leq \tilde{\delta}(s)$ for all $n \in \mathbb{N}$. The definition \eqref{eq:Moserdelta} and inequality \eqref{eq:auxEq48} imply $\tilde{\delta}(s) \geq \delta(s)$ also on $\{\tilde{\delta}(s) < \rho\}$. We conclude that $\P^{\mu_0}$-a.s., $\tilde{\delta}(s) = \delta(s)$ and the sequence of stopping times $\{\delta^{(\frac{1}{n})}(s)\}_{n \in \mathbb{N}}$ increases monotonically to $\delta(s)$. Hence, $\delta^{(0)}=\delta$ is a stopping time and by quasi-left continuity of $L$, \cite[Chap.~4, Thm.~3.12]{Ethier1986a} implies
  \begin{equation*} 
    \lim_{n \to \infty} L_{\delta^{(\frac{1}{n})}(s)} = L_{\delta(s)} , \quad \P^{\mu_0}\text{-a.s.} 
  \end{equation*}
  In particular, for any $f \in C_0(\mathbb{R})$,
  \begin{equation*} 
    \E^{\mu_0}[f(L_{\delta(s)})] = \lim_{n \to  \infty} \E^{\mu_0}[f(L_{\delta^{(\frac{1}{n})}(s)} )]=  \lim_{n \to  \infty} \int_\mathbb{R} f(x) \phi^{(\frac{1}{n})}(s,x) \dd x = \int_\mathbb{R} f(x) \phi(s,x) \dd x,
  \end{equation*}
  which implies that (ii) (and (iii), as seen from \eqref{eq:auxEq54}) has been established also without the assumption that $\sigma$ is bounded.
\end{proof}

\subsection*{Acknowledgement} 
D.J.P. gratefully acknowledges financial support of the Swiss National Foundation under Grant No.~$200021\_163014$. L.G. and D.J.P acknowledge generous support from ETH Z\"urich, where a major part of this work was completed

\bibliography{quellenSkoro}

\providecommand{\bysame}{\leavevmode\hbox to3em{\hrulefill}\thinspace}
\providecommand{\MR}{\relax\ifhmode\unskip\space\fi MR }
\providecommand{\MRhref}[2]{%
  \href{http://www.ams.org/mathscinet-getitem?mr=#1}{#2}
}
\providecommand{\href}[2]{#2}
\begin{thebibliography}{CGMY04}

\bibitem[AKU16]{Ankirchner2016}
Stefan Ankirchner, Thomas Kruse, and Mikhail Urusov, \emph{{Numerical
  approximation of irregular {SDE}s via {S}korokhod embeddings}}, J. Math.
  Anal. Appl. \textbf{440} (2016), no.~2, 692--715.

\bibitem[Bas83]{Bass1983}
Richard~F. Bass, \emph{{Skorokhod imbedding via stochastic integrals}},
  {Seminar on probability, {XVII}}, {Lecture Notes in Math.}, vol. 986,
  Springer, Berlin, 1983, pp.~221--224.

\bibitem[BC74]{baxter1974}
J.~R. Baxter and R.~V. Chacon, \emph{Potentials of stopped distributions},
  Illinois J. Math. \textbf{18} (1974), no.~4, 649--656.

\bibitem[Ber96]{Bertoin1996}
Jean Bertoin, \emph{L\'evy processes}, Cambridge Tracts in Mathematics, vol.
  121, Cambridge University Press, Cambridge, 1996.

\bibitem[BL92]{Bertoin1992}
Jean Bertoin and Yves {Le Jan}, \emph{{Representation of measures by balayage
  from a regular recurrent point}}, Ann. Probab. \textbf{20} (1992), no.~1,
  538--548.

\bibitem[CGMY04]{Carr2004}
Peter Carr, H\'{e}lyette Geman, Dilip~B. Madan, and Marc Yor, \emph{{From local
  volatility to local {L}\'{e}vy models}}, Quant. Finance \textbf{4} (2004),
  no.~5, 581--588.

\bibitem[CHO11]{Cox2011}
Alexander M.~G. Cox, David Hobson, and Jan Ob\l{\'o}j, \emph{Time-homogeneous
  diffusions with a given marginal at a random time}, ESAIM Probab. Stat.
  \textbf{15} (2011), no.~In honor of Marc Yor, suppl., S11--S24.

\bibitem[DGPR18]{Nr2}
Leif D{\"o}ring, Lukas Gonon, David~J. Pr\"omel, and Oleg Reichmann,
  \emph{{Existence and uniqueness results for time-inhomogeneous time-change
  equations and Fokker-Planck equations}}, Preprint ArXiv:1812.08579 (2018).

\bibitem[Dup94]{Dupire1994}
Bruno Dupire, \emph{{Pricing with a Smile}}, Risk Magazine (1994), 18--20.

\bibitem[EHJT13]{Ekstrom2013}
Erik Ekstr{\"o}m, David Hobson, Svante Janson, and Johan Tysk, \emph{Can
  time-homogeneous diffusions produce any distribution?}, Probab. Theory
  Related Fields \textbf{155} (2013), no.~3-4, 493--520.

\bibitem[EK86]{Ethier1986a}
Stewart~N. Ethier and Thomas~G. Kurtz, \emph{{Markov processes.
  Characterization and convergence}}, John Wiley \& Sons, 1986.

\bibitem[FF91]{Falkner1991}
Neil Falkner and P.~J. Fitzsimmons, \emph{Stopping distributions for right
  processes}, Probab. Theory Related Fields \textbf{89} (1991), no.~3,
  301--318.

\bibitem[FH16]{Feng2016}
Han Feng and David Hobson, \emph{{Gambling in contests with random initial
  law}}, Ann. Appl. Probab. \textbf{26} (2016), no.~1, 186--215.

\bibitem[GMO15]{Gassiat2015}
Paul Gassiat, Aleksandar Mijatovi\'{c}, and Harald Oberhauser, \emph{{An
  integral equation for {R}oot's barrier and the generation of {B}rownian
  increments}}, Ann. Appl. Probab. \textbf{25} (2015), no.~4, 2039--2065.

\bibitem[Hob98]{Hobson1998}
David~G. Hobson, \emph{Robust hedging of the lookback option}, Finance and
  Stochastics \textbf{2} (1998), no.~4, 329--347.

\bibitem[Hob11]{Hobson2011}
David Hobson, \emph{{The {S}korokhod embedding problem and model-independent
  bounds for option prices}}, {Paris-{P}rinceton {L}ectures on {M}athematical
  {F}inance 2010}, {Lecture Notes in Math.}, vol. 2003, Springer, Berlin, 2011,
  pp.~267--318.

\bibitem[HPRY11]{Hirsch2011}
Francis Hirsch, Christophe Profeta, Bernard Roynette, and Marc Yor,
  \emph{{Peacocks and associated martingales, with explicit constructions}},
  {Bocconi \& Springer Series}, vol.~3, Springer, Milan, 2011.

\bibitem[Kyp14]{Kyprianou2014}
Andreas~E. Kyprianou, \emph{Fluctuations of {L}\'evy processes with
  applications}, second ed., Universitext, Springer, Heidelberg, 2014,
  Introductory lectures.

\bibitem[Mon72]{Monroe1972}
Itrel Monroe, \emph{{Using additive functionals to embed preassigned
  distributions in symmetric stable processes}}, Trans. Amer. Math. Soc.
  \textbf{163} (1972), 131--146.

\bibitem[Ob{\l}04]{Obloj2004}
Jan Ob{\l}\'{o}j, \emph{{The {S}korokhod embedding problem and its offspring}},
  Probab. Surv. \textbf{1} (2004), 321--390.

\bibitem[OP09]{Obloj2009}
Jan Ob{\l}\'{o}j and Martijn Pistorius, \emph{{On an explicit {S}korokhod
  embedding for spectrally negative {L}\'{e}vy processes}}, J. Theoret. Probab.
  \textbf{22} (2009), no.~2, 418--440.

\bibitem[Ros70]{Rost1970}
Hermann Rost, \emph{Die {S}toppverteilungen eines {M}arkoff-prozesses mit
  lokalendlichem {P}otential}, manuscripta mathematica \textbf{3} (1970),
  no.~4, 321--329.

\bibitem[{Ros}71]{Rost1971}
Hermann {Rost}, \emph{{The stopping distributions of a Markov process}},
  {Invent. Math.} \textbf{14} (1971), 1--16.

\bibitem[Rud87]{Rudin1987}
Walter Rudin, \emph{{Real and complex analysis}}, third ed., McGraw-Hill Book
  Co., New York, 1987.

\bibitem[RW00a]{Rogers2000v2}
L.~C.~G. Rogers and David Williams, \emph{{Diffusions, Markov Processes, and
  Martingales}}, second ed., vol.~2, Cambridge University Press, 2000.

\bibitem[RW00b]{Rogers2000}
\bysame, \emph{{Diffusions, Markov Processes, and Martingales}}, second ed.,
  vol.~1, Cambridge University Press, 2000.

\bibitem[Sat72]{Sato1972}
Ken-Iti Sato, \emph{{Potential operators for Markov processes}}, {Proceedings
  of the Sixth Berkeley Symposium on Mathematical Statistics and Probability,
  Volume 3: Probability Theory} (Berkeley, Calif.), University of California
  Press, 1972, pp.~193--211.

\bibitem[Sat99]{Sato1999}
\bysame, \emph{{L\'{e}vy processes and infinitely divisible distributions}},
  Cambridge University Press, 1999.

\bibitem[Sko61]{Skorokhod1961}
A.V. Skorohod, \emph{{Issledovanija po teorii slu\v{c}ajnyh processov:
  (stohasti\v{c}askie differencial'nye uravnenija i predel'nye teoremy dlja
  processov Markova)}}, Izdatel'stvo Kievskogo universiteta, 1961 (Russian).

\bibitem[Sko65]{Skorokhod1965}
A.~V. Skorokhod, \emph{{Studies in the theory of random processes}},
  {Translated from the Russian by Scripta Technica, Inc}, Addison-Wesley
  Publishing Co., Inc., Reading, Mass., 1965.

\bibitem[SS13]{Seel2013}
Christian Seel and Philipp Strack, \emph{{Gambling in contests}}, Journal of
  Economic Theory \textbf{148} (2013), no.~5, 2033--2048.

\end{thebibliography}
\bibliographystyle{amsalpha}

\end{document}